\documentclass[11pt]{article}
\usepackage{latexsym,amsfonts,amssymb,amsmath,amsthm,fullpage,comment}
\usepackage{hyperref}
\usepackage{graphicx}

\usepackage[usenames,dvipsnames]{color}
\usepackage{ulem}

\newcommand {\rd}{\color{red}}
\newcommand{ \bk}{\color{black}}

\begin{document}
\setlength{\baselineskip}{14pt}

\parindent 0.5cm

\newtheorem{theorem}{Theorem}[section]
\newtheorem{lemma}{Lemma}[section]
\newtheorem{proposition}{Proposition}[section]
\newtheorem{definition}{Definition}[section]
\newtheorem{example}{Example}[section]
\newtheorem{corollary}{Corollary}[section]

\newtheorem{remark}{Remark}[section]

\numberwithin{equation}{section}

\def\p{\partial}
\def\I{\textit}
\def\R{\mathbb R}
\def\C{\mathbb C}
\def\u{\underline}
\def\l{\lambda}
\def\a{\alpha}
\def\O{\Omega}
\def\e{\epsilon}
\def\ls{\lambda^*}
\def\D{\displaystyle}
\def\wyx{ \frac{w(y,t)}{w(x,t)}}
\def\imp{\Rightarrow}
\def\tE{\tilde E}
\def\tX{\tilde X}
\def\tH{\tilde H}
\def\tu{\tilde u}
\def\d{\mathcal D}
\def\aa{\mathcal A}
\def\DH{\mathcal D(\tH)}
\def\bE{\bar E}
\def\bH{\bar H}
\def\M{\mathcal M}
\renewcommand{\labelenumi}{(\arabic{enumi})}

\def\disp{\displaystyle}
\def\undertex#1{$\underline{\hbox{#1}}$}
\def\card{\mathop{\hbox{card}}}
\def\sgn{\mathop{\hbox{sgn}}}
\def\exp{\mathop{\hbox{exp}}}
\def\OFP{(\Omega,{\cal F},\PP)}
\newcommand\JM{Mierczy\'nski}
\newcommand\RR{\ensuremath{\mathbb{R}}}
\newcommand\CC{\ensuremath{\mathbb{C}}}
\newcommand\QQ{\ensuremath{\mathbb{Q}}}
\newcommand\ZZ{\ensuremath{\mathbb{Z}}}
\newcommand\NN{\ensuremath{\mathbb{N}}}
\newcommand\PP{\ensuremath{\mathbb{P}}}
\newcommand\abs[1]{\ensuremath{\lvert#1\rvert}}

\newcommand\normf[1]{\ensuremath{\lVert#1\rVert_{f}}}
\newcommand\normfRb[1]{\ensuremath{\lVert#1\rVert_{f,R_b}}}
\newcommand\normfRbone[1]{\ensuremath{\lVert#1\rVert_{f, R_{b_1}}}}
\newcommand\normfRbtwo[1]{\ensuremath{\lVert#1\rVert_{f,R_{b_2}}}}
\newcommand\normtwo[1]{\ensuremath{\lVert#1\rVert_{2}}}
\newcommand\norminfty[1]{\ensuremath{\lVert#1\rVert_{\infty}}}

\title{ Asymptotic limits of the principal spectrum point of  a nonlocal dispersal cooperative system and application to a two-stage structured population model} 
\author{Maria A. Onyido\footnote{Department of Mathematical Sciences, Northern Illinois University, DeKalb, IL 60115, USA (monyido@niu.edu). }, 
\quad Rachidi B. Salako\footnote{Department of Mathematical Sciences, University of Nevada  Las Vegas, NV 89154, USA
(rachidi.salako@unlv.edu).},\quad  Markjoe O. Uba\footnote{Department of Mathematical Sciences, Northern Illinois University, DeKalb, IL 60115, USA (markjoeuba@gmail.com).}, \quad and \quad  Cyril I. Udeani\footnote{Department of Applied Mathematics and Statistics, Comenius University in Bratislava, Mlynsk\'a dolina, 84248 Bratislava, Slovakia (cyril.izuchukwu@fmph.uniba.sk).} }
\date{}

\maketitle

\begin{abstract}This work examines the limits of the principal spectrum point, $\lambda_p$, of a nonlocal dispersal cooperative system with respect to the dispersal rates. In particular, we provide precise information on the sign of $\lambda_p$ as one of the dispersal rates is : (i) small while the other dispersal rate is arbitrary, and (ii)  large while the other is either also large or  fixed. We then apply our results to study the effects of dispersal rates on a two-stage structured nonlocal dispersal population model whose linearized system at the trivial solution results in a nonlocal dispersal cooperative system. The asymptotic profiles of the steady-state solutions with respect to the dispersal rates of the two-stage nonlocal dispersal population model are also obtained. Some biological interpretations of our results are discussed.
    
\end{abstract}

\noindent{\bf Key words.}  Nonlocal-dispersal, stage-structured model, principal spectrum point, steady state, asymptotic limit

\smallskip

{\noindent{\bf  AMS subject classifications.} 92D40, 92D50, 35P15, 35K57

\section{Introduction}
\quad Most population models typically investigate species under the assumption that they share similar characteristics among its kind. However, biological species could exhibit non-negligible stage-specific variations. Therefore, a stage-structured model would help in incorporating stage-dependent physiological parameters, thereby yielding better biological predictions of survival or extinction. For example, many  insect  species pass from an egg stage into several instar stages and then into the adult stage. For some studies on stage structured populations, we refer  interested readers to \cite{BWH2019, BM1990} and the references therein.

\bk In this study, we consider the stage-structured population model with nonlocal dispersal, introduced in \cite{OSUU2023}, describing the dynamics of a species with density function ${\bf u}(t,x)=(u_1(t,x),u_2(t,x))$ living in a bounded habitat $\Omega\subset\mathbb{R}^n$ and structured in two stages: adults and juveniles. Here, $u_1(t,x)$ denotes the density function of the juveniles and $u_2(t,x)$ that of the adults who have attained reproductive maturity. The adults have local reproductive rate $r(x)$, and the juveniles attain reproductive maturity at the rate $s(x)$. The juveniles have local death rate $a(x)$ and self-limitation rate $b(x)$ due to their size in space. The adults have a local death rate $e(x)$ and local self-limitation rate $f(x)$ induced by their size. Interaction between the adults and the juveniles may generate an interspecific competition for local resources. We denote by $c(x)$ and $g(x)$ the interspecific local competition rates of the adults and juveniles, respectively.   The following nonlocal system of PDE can be used to study the dynamics of the species
\begin{small}
 \begin{equation}\label{model}
     \begin{cases}
     \partial_tu_1=\mu_1\int_{\Omega}\kappa(x,y)(u_1(t,y)-u_1(t,x))dy +r(x)u_2-s(x)u_1 -(a(x)+b(x)u_1+ c(x)u_2)u_1 & x\in\Omega,\ t>0,\cr
      \partial_tu_2=\mu_2\int_{\Omega}\kappa(x,y)(u_2(t,y)-u_2(t,x))dy +s(x)u_1-(e(x)+f(x)u_2+ g(x)u_1)u_2 & x\in\Omega,\ t>0,
     \end{cases}
 \end{equation}
 \end{small}
\noindent where $\mu_1$ and $\mu_2$ represent the dispersal rates.   The nonlocal dispersal operator in \eqref{model} is said to have Neumann-type boundary condition  (see \cite{CERW} for more details). \bk   Nonlocal dispersal, modeled by an integral operator like that appearing in system \eqref{model},  is employed in modeling the dynamics of species that exhibit long-range movement or position jump  (see \cite{BL, BS, CCR2006, CERW, Co1, OSUU2023, OnSh, SX2015, ShZh1, SV} and the references therein for more discussion on population models with nonlocal dispersal).   
 Throughout this study, we shall suppose that the following standing hypotheses hold for the parameters in system \eqref{model}:   
 
 \medskip

 \noindent {\bf (H1)}   $\kappa\in C(\bar{D}\times\bar{D})$, positive and symmetric, i.e., $\kappa(x,y)=\kappa(y,x)>0$ for every $ x,y\in\bar{D}$.  
 
 \medskip

 \noindent {\bf (H2)} The functions  $a,b, c, e, f, g, r,$ and $s$ are  H\"older continuous on $\overline{\Omega}$ and nonnegative, $r$ and $s$ are non-identically zero, and  $b$ and $f$ are strictly positive.

\medskip

 Observe that ${\bf u}(t,x) = (0,0)$ is a soluton of \eqref{model}, known as the trivial solution. Linearizing \eqref{model} at this trivial solution yields the following linear cooperative nonlocal system:
 
\begin{equation}\label{linera-model}
     \begin{cases}
     \partial_tU_1=\mu_{1}\int_{\Omega}\kappa(x,y)[U_1(t,y)-U_1(t,x)]dy +r(x)U_2-(a(x)+s(x))U_1 & x\in \Omega,\ t>0,\cr 
     \partial_tU_2=\mu_2\int_{\Omega}\kappa(x,y)[U_2(t,y)-U_2(t,x)]dy +s(x)U_1-e(x)U_2 & x\in\Omega,\ t>0.
     \end{cases}
 \end{equation}
 Note that using equations  \eqref{kappa-def}, \eqref{circ-def} and \eqref{A-map} below, \eqref{linera-model} can be rewritten as 

 \begin{equation}\label{M-eq3}
     \frac{d{\bf U}}{dt}=\mu\circ\mathcal{K}({\bf U})+\mathcal{A}{\bf U}\quad t>0.
 \end{equation}  
Denoting by  $\lambda_p(\mu\circ\mathcal{K}+\mathcal{A})$  the principal spectrum point of the linear operator $\mu\circ\mathcal{K}+\mathcal{A}$ (see Definition \eqref{def-0}). It was established in
\cite{OSUU2023}  that the persistence and extinction of the species modeled by \eqref{model}  depend  on the sign of  $\lambda_p(\mu\circ\mathcal{K}+\mathcal{A})$.   Hence, to understand the effect of the dispersal rates on the species' persistence, it is pertinent to investigate the asymptotic dynamics of the principal spectrum point with respect to the dispersal rates. Our first goal in this study is to examine some characterizations of the principal spectrum point and determine its behavior as the dispersal rates approach some critical values. 

In general, cooperative nonlocal dispersal models appear in different scenarios  (see\cite{H, H2} and the references therein),  like the general two-species  Lotka-Volterra competition system  (see \cite{BS} and the references therein). We indicate that two-species competition systems can be transformed into cooperative ones using standard arguments. Hence, our results in the current work on $\lambda_p$ apply to that setting. For instance, \cite{BS} established some criteria for the existence of principal eigenvalues of time-periodic cooperative linear systems with nonlocal dispersal. Since the principal eigenvalues (when it exists) coincide with the principal spectrum point, our results in this study also apply to their study when the coefficients are time-homogeneous.

 For unstructured species modeled with nonlocal dispersal under  a temporally homogeneous but spatially heterogeneous environment with Neumann/Dirichlet-type boundary conditions, Shen and Xie \cite{SX2015} examined the effects of the dispersal rates on the asymptotic behavior of the principal spectrum point of its linearization at the trivial solution. They obtained the limits of the principal spectrum point as the dispersal rates go to zero or infinity. Hence, our results on the limits of  $\lambda_p(\mu\circ\mathcal{K}+\mathcal{A})$ (see Theorem \ref{Th2-1}) extend  their results to two species cooperative case. Additionally, Shen and Vo \cite{SV} studied the unstructured model with Dirichlet-type boundary conditions and extended some of the  results  of Shen and Xie to a time-periodic environment. 

It is pertinent to mention at this point that \eqref{model} is the nonlocal analog  of the following two-stage structured population system with random (local) dispersal and Neumann boundary conditions \bk
\begin{equation}\label{local}
     \begin{cases}
      \partial_tu_1=\mu_1\Delta u_1 +r(x)u_2-s(x)u_1 -(a(x)+b(x)u_1+c(x)u_2)u_1 & x\in\Omega,\ t>0,\cr
      \partial_tu_2=\mu_2\Delta u_2 +s(x)u_1-(e(x)+f(x)u_2+g(x)u_1)u_2 & x\in\Omega,\ t>0,\cr
      \nabla u_1 \cdot \nu =  \nabla u_2 \cdot \nu = 0 \; & x \in \partial \Omega, \; t > 0,
     \end{cases}
 \end{equation}
where $\nu$ is the outward unit normal to $\partial \Omega$  and the parameters in \eqref{local} have the same meaning as those of \eqref{model}. Cantrell, Cosner and Salako \cite{CCS} studied the effects of dispersal rates on the species modeled by \eqref{local} and consequently obtained some results on the asymptotic behavior of solutions as the dispersal rates go to zero or infinity. Our results extend theirs to the nonlocal case. It is noteworthy to indicate that the semiflow generated by the classical solution of system \eqref{model} is not compact. This noncompactness introduces some difficulties in  the analysis of the nonlocal system and most of the techniques developed for \eqref{local} do not generally extend to \eqref{model}.

 When species persist and eventually stabilize, from an  ecological viewpoint, it is important to examine their spatial distributions for adequate distribution of scarce resources. In this regard, the second goal of this study is to apply the results obtained in the first part to establish the asymptotic profiles of steady states of system \eqref{model}. In particular, our results in Theorems \rd \ref{Th2-5} \bk provide complete information on the spatial distribution of steady states of \eqref{model} when at least one of the dispersal rates go to zero or infinity. It is appropriate to note here that \cite{OSUU2023} gave some criteria for the existence, uniqueness and stability of steady states of \eqref{model}. Hence, our results in this regard complement those results.

  The rest of the paper is organized as follows:  Section \ref{Sec2} contains some notations, definitions and  our main results. We collect a few preliminary results  in  section \ref{Sec3}. The proofs of our main results are presented in sections \ref{Sec4}--\ref{Sec8}.

\bk
\section{Notations, Definitions and Main Results} \label{Sec2}
\subsection{Notations and Definitions}

 \quad Let $X:=C(\overline{\Omega})$ denote the Banach space of uniformly continuous and bounded functions on $\overline{\Omega}$ endowed with the sup-norm, $\|u\|_{\infty}=\underset{x\in \Omega}{\sup}|u(x)|$ for every $u\in X$. Since the density functions are nonnegative, we will be interested in the following subsets of $X$: 
$$
X^+=\{u\in X\,|\, u(x)\ge 0,\quad x\in \bar \Omega\},
$$
and
$$
X^{++}=\{u\in X^+\,|\, \inf_{x\in\bar \Omega} u(x)>0\}.
$$
We shall use bold-face letters to represent vectors. In particular,  we use ${\bf u}=(u_1,u_2)$  to denote vectors in $\mathbb{R}^2$. For covenience, we also use  ${\bf u}=(u_1,u_2)$ for elements of $X\times X$. Hence, any vector in $\mathbb{R}^2$ can be understood as a constant vector valued function on $\overline{\Omega}$; this would not cause any confusion in the paper. We endow  $X\times X$ with the norm 
$$ 
\|{\bf u}\|:=\max\{\|u_1\|_{\infty},\|u_2\|_{\infty}\}\quad \forall\ {\bf u}\in X\times X.
$$
 Hence, $X\times X$ is also a Banach space.  Given a function $h\in X$, define
$$ 
\hat{h}:=\frac{1}{|\Omega|}\int_{\Omega}h(x)dx,\quad h_{\min}:=\min_{x\in\overline{\Omega}}h(x), \quad \text{and}\quad  h_{\max}=\max_{x\in\overline{\Omega}}h(x).
$$
 Given a complex number $\lambda\in\mathbb{C}$, we denote by $\mathcal{R}e(\lambda)$ its real part.

\begin{definition}[Principal Spectrum Point]\label{def-0} Let $E$ be a Banach space and $B : dom(B)\to E$ be a linear map where $dom(B)$ is a linear subspace of $E$. Let $\sigma(B)$ denote the spectrum of 
 the linear map $B$. The principal spectrum point of $B$, denoted  $\lambda_p(B)$, is defined by 
  $$ 
  \lambda_p(B):=
  \begin{cases}
  \sup\{\mathcal{ R}e(\lambda) : \lambda\in\sigma(B) \} & \text{if}\ \sigma(B)\ne \emptyset,\cr 
  -\infty & \text{if}\ \sigma(B)= \emptyset.
  \end{cases}
  $$ 
\end{definition}

\medskip

\noindent For our purpose in the current work, the Banach space in Definition \ref{def-0} will be either $E=X$  or $E=X\times X$, while the linear operator $B$ will always be a bounded linear map on $E$. Hence,  $\sigma(B)$ will always be nonempty and bounded; so that the principal spectrum point $\lambda_p(B)$ is a real number. 

\medskip 

 Given $h\in X$, consider the bounded linear operators $\mathcal{K},\ \mathcal{I},\ h\mathcal{I}\ :\ X\to X$ defined by 
\begin{equation}\label{kappa-def}
    \mathcal{K}u(x)=\int_{\Omega}\kappa(x,y)(u(y)-u(x))dy \quad \forall\ u\in\ X,\ x\in\overline{\Omega},
\end{equation}
\begin{equation}\label{Ident-def}
    \mathcal{I}u(x)=u(x)\quad \forall\ u\in X,\ x\in\overline{\Omega},
\end{equation}
\begin{equation}\label{multi-func}
    (h\mathcal{I})u(x)=h(x)u(x)\qquad \forall\ u\in X,\ x\in\overline{\Omega}.
\end{equation}

\medskip

\noindent Next, given a bounded linear operator $\mathcal{B}$ on $X$ and a real vector ${\bf z}=(z_1,z_2)\in\mathbb{R}^2$, we denote by ${\bf z}\circ \mathcal{B}$ the bounded linear operator on $X\times X$ given by 
\begin{equation}\label{circ-def}
    {\bf z}\circ \mathcal{B}{\bf u}=\left(\begin{array}{c}
         z_1\mathcal{B}u_1  \\
         z_2\mathcal{B} u_2  
    \end{array} \right) \quad \forall\ {\bf u}\in X\times X.
\end{equation}

\noindent Define the operator $ \mathcal{A} \ :\ X\times X\to X\times X$ by     
\begin{equation}\label{A-map}    \mathcal{A}({\bf u})=A(x){\bf u}=\left( \begin{array}{c}       r(x)u_2-(a(x)+s(x))u_1      \\    su_1(x)-e(x)u_2(x)           \end{array}\right)\quad \forall\ {\bf u}\in X\times X.\end{equation} where \begin{equation}\label{matrix-A}
     A(x)=\left(\begin{array}{cc}
         -(a(x)+s(x)) & r(x)  \\
         s(x) & -e(x)
     \end{array}\right) \quad  \forall x\in\overline{\Omega}.
 \end{equation}

 For convenience, we introduce the matrix
\begin{equation}\label{A-hat}
\hat{A}=\left(\begin{array}{cc}
         -(\hat{a}+\hat{s}) &  \hat{r}  \\
         \hat{s} & -\hat{e}
     \end{array}\right),
     \end{equation}
and let $\tilde{\Lambda}$ denote its maximal eigenvalue.

\subsection{Main Results }

\quad We state our main results in this  section. For every positive vector ${\bf \mu}=(\mu_1,\,\mu_2)\in \mathbb{R}^+\times\mathbb{R}^+$,  let $\sigma(\mu\circ\mathcal{K}+\mathcal{A})$ denote the spectrum of the bounded linear operator $\mu\circ\mathcal{K}+\mathcal{A}$ and $\lambda_p(\mu\circ\mathcal{K}+\mathcal{A})$ its principal spectrum point. If $\lambda_p(\mathcal{\mu\circ\mathcal{K}+\mathcal{A}})$ is an isolated eigenvalue with a strictly positive eigenfunction, then $\lambda_p(\mathcal{\mu\circ\mathcal{K}+\mathcal{A}})$ is called the principal eigenvalue of $\mathcal{\mu\circ\mathcal{K}+\mathcal{A}}$. However, in general, it is well-known that  $\lambda_p(\mathcal{\mu\circ\mathcal{K}+\mathcal{A}})$ may  not be  the principal eigenvalue (see \cite{Co1, ShZh1}  and the references therein). Thus, it is essential to find some other useful ways to characterize $\lambda_p(\mathcal{\mu\circ\mathcal{K}+\mathcal{A}})$ by comparing it to some other relevant spectral quantities. This is accomplished in the next subsection.

\subsubsection{Characterization of $\lambda_p(\mu\circ\mathcal{K}+\mathcal{A})$ in terms of generalized principal eigenvalues}

\quad Our first result in this subsection provides some useful characterizations of the principal spectrum point $\lambda_p (\mathcal{\mu \circ \mathcal{K}+\mathcal{A}})$ for two-species nonlocal equations.  Consider the following ordering in $X \times X$
\begin{equation}\label{ord1}
    {\bf u} \le_1  {\bf v}\;  \text{if}\; u_1 \le v_1 \ \text{and}\ u_2 \le  v_2.
\end{equation} 
\noindent Note that for the positive constant vector ${\bf 1}:=(1,1)$, we have  $(\mu\circ\mathcal{K}+\mathcal{A})({\bf 1})=A(x){\bf 1}$ and 
 $$
 -\max\{\|a+s\|_{\infty},\|e\|_{\infty}\}{\bf 1}  \le_1  A(x){\bf 1} \le_1 \bk \max\{\|s\|_{\infty},\|r\|_{\infty}\}{\bf 1}\quad \forall\ x\in\overline{\Omega}.
 $$ \bk
 Therefore, the quantities 
 \begin{equation}\label{HK1}
     \lambda_*(\mu\circ\mathcal{K}+\mathcal{A}):=\sup\{\lambda\in\mathbb{R}\ :\ \exists \varphi\in X^{++}\times X^{++} \ \text{satisfying}\ \lambda\varphi \le_1 (\mu\circ\mathcal{K}+\mathcal{A})(\varphi) \}
 \end{equation}
 and 
 \begin{equation}\label{HK2}
    \lambda^*(\mu\circ\mathcal{K}+\mathcal{A}):=\inf\{\lambda\in\mathbb{R}\ :\ \exists \varphi\in X^{++}\times X^{++} \ \text{satisfying}\ (\mu\circ\mathcal{K}+\mathcal{A})(\varphi)\le_1 \lambda\varphi\}
 \end{equation}
 are well-defined real numbers and satisfy 
 \begin{equation}\label{HK3}
     -\max\{\|a+s\|_{\infty},\|e\|_{\infty}\}\le  \lambda_*(\mu\circ\mathcal{K}+\mathcal{A})\le  \lambda^*(\mu\circ\mathcal{K}+\mathcal{A})  \le \max\{\|r\|_{\infty},\|s\|_{\infty}\}. 
 \end{equation}
 
\noindent The fact that $\lambda_*(\mu\circ\mathcal{K}+\mathcal{A}) \le \lambda^*(\mu\circ\mathcal{K}+\mathcal{A})$ is a consequence of the positivity of the uniformly continuous semigroup $\{{\bf U}^{\mu}(t)\}_{t\ge 0}$ generated by the bounded linear operator $\mu\circ\mathcal{K}+\mathcal{A}$. The quantities $\lambda_*(\mu\circ\mathcal{K}+\mathcal{A})$
 and $\lambda^*(\mu\circ\mathcal{K}+\mathcal{A})$ are the {\it generalized principal eigenvalues} of the linear operator $\mu\circ\mathcal{K}+\mathcal{A}$.  The following result holds.

 \begin{theorem}\label{TH4} Let  $\lambda_*(\mu\circ\mathcal{K}+\mathcal{A})$ and   $\lambda^*(\mu\circ\mathcal{K}+\mathcal{A})$ be defined by \ \eqref{HK1} and \eqref{HK2}, respectively. Then 
 $$ 
 \lambda_*(\mu\circ\mathcal{K}+\mathcal{A})=  \lambda^*(\mu\circ\mathcal{K}+\mathcal{A})=\lambda_p(\mu\circ\mathcal{K}+\mathcal{A}).
 $$
 \end{theorem}

 Theorem \ref{TH4} shows that $\lambda_p(\mathcal{\mu\circ\mathcal{K}+\mathcal{A}})$ equals the generalized principal eigenvalues $\lambda_*(\mathcal{\mu\circ\mathcal{K}+\mathcal{A}})$ and $\lambda^*(\mathcal{\mu\circ\mathcal{K}+\mathcal{A}})$; hence, it extends the known result for the single species nonlocal equations (see \cite{OnSh, BCV})  to two species nonlocal equations. It would be of great interest to know whether the conclusions of Theorem \ref{TH4} hold in time-periodic and bounded spatially heterogeneous environments or time-space periodic environments. We hope to address these questions in our future works.  The characterization of the principal spectrum point as in Theorem \ref{TH4} turns out to be an important tool in the proofs of some of our main results on the limit of principal spectrum point with respect to the dispersal rates. The next subsection discusses our main results  in this regard.

\subsubsection{ Asymptotic limits of $\lambda_p(\mu\circ\mathcal{K}+\mathcal{A})$ with respect to $\mu$.}

\quad As established in \cite[Theorem 1]{OSUU2023}, the sign of $\lambda_p(\mu\circ\mathcal{K}+\mathcal{A})$ completely determines the persistence of the species. Hence, it is of great biological interest to understand how the dispersal rate $\mu$ affect $\lambda_p(\mu\circ\mathcal{K}+\mathcal{A})$  since the influence of the dispersal rates on the principal spectrum point translates to the effects of dispersal rates on the species' persistence and extinction.  In this section, we present several results on the limit of $\lambda_p(\mu\circ\mathcal{K}+\mathcal{A})$ as $\mu$ approaches some critical values. These results will help to precisely determine the  behavior  of $\lambda_p(\mu\circ\mathcal{K}+\mathcal{A})$ in several instances.

\quad Our first result in this section concerns the limit of $\lambda_p(\mu\circ\mathcal{K}+\mathcal{A})$ for small $\mu$ and large $\mu$. 

\begin{theorem}\label{Th2-1}
\begin{itemize}
    \item[\rm (i)]  For each $x\in\overline{\Omega}$, let $\Lambda(x)$  denote the maximal eigenvalue of the cooperative matrix $A(x)$. Then
\begin{equation}\label{Th2-1-Eq}
    \lim_{\max\{\mu_1,\mu_2\}\to { 0}}\lambda_p(\mu\circ\mathcal{K}+\mathcal{A})=\Lambda_{\max}.
\end{equation}

\item[\rm (ii)] Let $\tilde{\Lambda}$ denote the maximal eigenvalue of the cooperative matrix $\hat{A}$. Then
\begin{equation}\label{Th2-2-eq}
\lim_{\min\{\mu_1,\mu_2\}\to\infty}\lambda_p(\mu\circ\mathcal{K}+\mathcal{A})=\tilde{\Lambda}.
\end{equation}
\end{itemize}

\end{theorem}

 Theorem \ref{Th2-1} provides the exact information on the limits of the principal spectrum point for large and small dispersal rates of the population. It is important to indicate that $\tilde{\Lambda}$ is the maximum eigenvalue of the cooperative matrix resulting from taking the spatial averages of the parameter functions in the matrix $A(x)$, which is different from the average of the function $\Lambda(x)$. This should not cause any confusion in the presentation. We provide some remarks on the implications of the  results.

\begin{remark}
\begin{description}
    \item[\rm (i)] Thanks to Theorem \ref{Th2-1}-{\rm (i)}, the species persists for small dispersal rates if and only if $\Lambda(x)>0$ for some $x\in\Omega$. Observe that $\Lambda(x)>0$ if and only if $r(x)s(x)>(s(x)+a(x))e(x)$.  Note that $(a(x)+s(x))e(x)$ (resp. $r(x)s(x)$ ) is the product of the loss (resp. growth) in the rates of change of the density functions of the juveniles and adults.   Hence, $\Lambda(x)>0$ indicates that, at such a location, the species are able to reproduce and grow while having a small death rate.  If the habitat  has such locations,  then small dispersal rates will benefit species' survival.

    \item[\rm (ii)] Theorem \ref{Th2-1}-{\rm (ii)} indicates that for large dispersal rates of the species, $\lambda_p(\mu\circ\mathcal{K}+\mathcal{A})$ approximate the maximal eigenvalue of the matrix $\hat{A}$ defined by the spatial averages of the parameters as defined by \eqref{A-hat}.   Observe that $\tilde{\Lambda}>0$ if and only if $\hat{r}\cdot\hat{s}-(\hat{a}+\hat{s})\hat{e}>0$. This shows that for  large dispersal rates, the persistence of the species depends on the spatial averages of the functions $a, s, r$, and $e$ and not on their local spatial distributions.  It is important to note from Theorem \ref{Th2-1}-(i), that for small dispersal rates, it is the local distributions of these functions that determine the persistence of the species. Hence, if for example the functions $r$ and $s$ have disjoint support with the product of their averages higher than  the product of the averages of $(a+s)$ and $e$, then the species will go extinct if dispersing very slowly  whereas they will persist if diffusing very fast. In such scenario, small dispersal rates lead to extinction while large dispersal rates leads to persistence. This is in contrast with the prediction for persistence of species modeled  with unstructured population models.
\end{description}

\end{remark}

Now, we discuss the scenarios where one of the dispersal rate is small and the other is fixed.   For clarity in the statement of our results in these scenarios, we first state the following proposition.

\begin{proposition}\label{Main-prop1} Let $h,l\in X^{+}$ and $\xi>0$ be given. The algebraic equation in $\nu$,
\begin{equation}\label{WE-1}
    0=\lambda_p\Big(\xi\mathcal{K}+\big(\frac{rs}{\nu+h}-l\big)\mathcal{I}\Big)-\nu, \quad \quad  \ \nu>-h_{\min},
\end{equation}
has a (unique)  zero, denoted by $\lambda_{\xi}^{h,l}$, if and only if 
\begin{equation}\label{WE-2}
    \sup_{\nu> -h_{\min}}\Big(\lambda_p\Big(\xi\mathcal{K}+\big(\frac{rs}{\nu+h}-l\big)\mathcal{I}\Big)-\nu\Big)>0.
\end{equation}
Furthermore, $\lambda_{\xi}^{h,l}$, whenever it exists, has the same sign as $ \lambda_p\Big(\xi\mathcal{K}+\big(\frac{rs}{h}-l\big)\mathcal{I}\Big)$ if $h_{\min}>0$.
\end{proposition}

\noindent Thanks to Proposition \eqref{Main-prop1}, we can now state our result on the limit of $\lambda_p(\mu\circ\mathcal{K}+\mathcal{A})$ as either $\mu_1\to0$ or $\mu_2\to 0$.
\begin{theorem}\label{Th2-3}
\begin{itemize}
\item[\rm (i)] Let $\xi=\mu_2$, $h=a+s$ and $l=e$ in Proposition \ref{Main-prop1}. Let  $\lambda_{\mu_2}^{a+s,e}$  denote the unique  zero  of the algebraic equation \eqref{WE-1} when inequality \eqref{WE-2} holds, and $\lambda_{\mu_2}^{a+s,e}:=-(a+s)_{\min}$ when \eqref{WE-2} doesn't hold. Then 
\begin{equation}
\lim_{\mu_1\to0^+}\lambda_p(\mu\circ\mathcal{K}+\mathcal{A})={\lambda}_{\mu_2}^{a+s,e}.
\end{equation}

\item[\rm (ii)]  Let $\xi= \mu_1$, \bk $h=e$ and $l=a+s$ in Proposition \ref{Main-prop1} . Let $\lambda_{\mu_1}^{e,a+s}$  denote the unique solution of the algebraic equation \eqref{WE-1} when inequality \eqref{WE-2} holds, and $\lambda_{\mu_1}^{e,a+s}:=-e_{\min}$ when \eqref{WE-2} doesn't hold. Then 
\begin{equation}
\lim_{\mu_2\to0^+}\lambda_p(\mu\circ\mathcal{K}+\mathcal{A})={\lambda}_{\mu_1}^{e,a+s}.
\end{equation}

\end{itemize}
\end{theorem}

\begin{remark}\label{R2-2} Let us assume  that $a+s$ is strictly positive on  $\overline{\Omega}$ and  fix the dispersal rate $\mu_2$ of the adult.
\begin{itemize}
\item[\rm (i)] Suppose that $\lambda_p\Big(\mu_2\mathcal{K}+\Big(\frac{rs}{a+s}-e\Big)\mathcal{I}\Big)>0$. It follows from Proposition \ref{Main-prop1} and Theorem \ref{Th2-3}-{\rm (i)} that   the species persists for small dispersal rate of the juvenile. Precisely, there is $\mu_{1,\mu_2}>0$ such that the species persists whenever $\mu_1<\mu_{1,\mu_2}$.  Observing that $\lambda_p\Big(\tilde{\mu}_2\mathcal{K}+\Big(\frac{rs}{a+s}-e\Big)\mathcal{I}\Big) $ is non-increasing in $\tilde{\mu}_2$,  (see \cite[Theorem 2.2(1)]{SX2015}),  then $\lambda_p\Big(\tilde{\mu}_2\mathcal{K}+\Big(\frac{rs}{a+s}-e\Big)\mathcal{I}\Big)>0$ for every $\tilde{\mu}_2<\mu_2$. Therefore, it becomes an interesting quest to know   whether the positive  number $\mu_{1,\mu_2}$ can be chosen such that the species persists whenever the dispersal rates $(\mu_1,\tilde{\mu}_2)$ satisfies $\mu_1<\mu_{1,\mu_2}$  and $\tilde{\mu}_2<\mu_2$. This question will be completely settled by Theorem \ref{Th2-6} below, which provides precise information on the sign of the principal spectrum point when the dispersal rates are near the $\mu_2$-axis. 

\item[\rm (ii)] If $\lambda_p\Big(\mu_2\mathcal{K}+\Big(\frac{rs}{a+s}-e\Big)\mathcal{I}\Big)<0$, it follows from Proposition \ref{Main-prop1} and Theorem \ref{Th2-3}-{\rm (i)} that   the species will go extinct if the juveniles disperse very slowly. In this case,  it is  unclear  whether fast movement of the juvenile would be beneficial for the species survival. So, a good understanding of the limit of the principal spectrum point for large dispersal of only the juvenile is essential. This will be investigated in the next result. 

\item[\rm (iii)] Observe  that $\lambda_p\Big(\tilde{\mu}_2\mathcal{K}+\Big(\frac{rs}{a+s}-e\Big)\mathcal{I}\Big) \to \Big(\frac{rs}{a+s}-e\Big)_{\max} $ as $\tilde{\mu}_2\to 0$. Note also that $\Big(\frac{rs}{a+s}-e\Big)_{\max}$ has the same sign as $\Lambda_{\max}$. Hence the results of Theorems \ref{Th2-1} and \ref{Th2-3} on the persistence of the species for small dispersal rates  are consistent. 
\end{itemize}
    
\end{remark}

 To discuss the limit of $\lambda_p(\mu\circ\mathcal{K}+\mathcal{A})$ as either $\mu_1\to\infty$ or $\mu_2\to\infty$, we will need the following intermediate result.

\begin{proposition}\label{Main-pop2}
     Fix $\xi>0$, $q,z\in X^{+}\setminus\{0\}$, and $l,h\in X^+$.   The algebraic equation in $\nu$,
    \begin{equation}\label{IL1}
        0=\int_{\Omega}q(x)(\nu\mathcal{I}-(\xi\mathcal{K}-l\mathcal{I}))^{-1}(z)(x)dx-\int_{\Omega}h(x)dx-\nu|\Omega|,\quad  \nu>\lambda_p(\xi\mathcal{K}-l\mathcal{I}) ,
    \end{equation}
     has a (unique) zero, denoted as $\tilde{\lambda}_{q,\xi,z}^{h,l}$, if and only if 
    \begin{equation}\label{IL2}
        \sup_{\nu>\lambda_p(\xi\mathcal{K}-l\mathcal{I})}\Big(\int_{\Omega}q(x)(\nu\mathcal{I}-(\xi\mathcal{K}-l\mathcal{I}))^{-1}(z)(x)dx-\int_{\Omega}h(x)dx-\nu|\Omega|\Big)>0.
    \end{equation}
    If \eqref{IL2} holds, then 
    \begin{equation}\label{Pl1}
        \begin{cases}
            \tilde{\lambda}_{q,\xi,z}^{h,l}|\Omega|= \int_{\Omega}q\psi -\int_{\Omega}h\cr 
            \tilde{\lambda}_{q,\xi,z}^{h,l}\psi=\xi\mathcal{K}(\psi)-l\psi +z
        \end{cases}
    \end{equation}
    where $\psi=(\tilde{\lambda}_{q,\xi,z}^{h,l}\mathcal{I}-(\xi\mathcal{K}-l\mathcal{I}))^{-1}$. 
    Furthermore,  the following  holds. 
    \begin{itemize}
        \item[\rm (i)] If  $\lambda_p(\xi\mathcal{K}-l\mathcal{I})$ is an eigenvalue of $\xi\mathcal{K}-l\mathcal{I}$ with a  positive eigenfunction, then \eqref{IL2} holds.

        \item[\rm (ii)] If $\lambda_p(\xi\mathcal{K}-l\mathcal{I})<0$ and \eqref{IL2} holds, the unique zero $\tilde{\lambda}_{q,\xi,p}^{h,l}$ has the same sign as the quantity $\sigma_{q,\xi,z}^{h,l}$ defined by 
        \begin{equation}\label{Pl2}
    \sigma_{q,\xi,z}^{h,l}:=\int_{\Omega}q(x)(l\mathcal{I}-\xi\mathcal{K})^{-1}(z)(x)dx-\int_{\Omega}h(x)dx.
    \end{equation}
    \end{itemize}

\end{proposition}

 \bk

The following result is concerned with the limit of $\lambda_p(\mu\mathcal{K}+\mathcal{A})$ as either $\mu_1\to\infty$ or $\mu_2\to\infty$.

\bk

\begin{theorem}\label{Th2-4}
\begin{itemize}

\item[\rm (i)]  Let $\xi=\mu_2$, $q=r$, $z=s$, $l=e$ and $h=a+s$ in Proposition \ref{Main-pop2}. Let $\tilde{\lambda}^{a+s,e}_{r,\mu_2,s}$ denote the unique root of \eqref{IL1} when \eqref{IL2} holds, and $\tilde{\lambda}_{r,\mu_2,s}^{a+s,e}=\lambda_p(\mu_2\mathcal{K}-e\mathcal{I})$ when \eqref{IL2} doesn't hold.   Then 
\begin{equation}
\lim_{\mu_1\to\infty}\lambda_p(\mu\circ\mathcal{K}+\mathcal{A})=\tilde{\lambda}^{a+s,e}_{r,\mu_2,s}.
\end{equation}

\item[\rm (ii)]  Let $\xi=\mu_1$, $q=s$, $z=r$, $l=a+s$ and $h=e$ in Proposition \ref{Main-pop2}. Let $\tilde{\lambda}^{e,a+s}_{s,\mu_1,r}$ denote the unique root of \eqref{IL1} when \eqref{IL2} holds, and $\tilde{\lambda}_{s,\mu_1,r}^{e,a+s}=\lambda_p(\mu_1\mathcal{K}-(a+s)\mathcal{I})$ when \eqref{IL2} doesn't hold. 
 Then  
\begin{equation}
\lim_{\mu_1\to\infty}\lambda_p(\mu\circ\mathcal{K}+\mathcal{A})=\tilde{\lambda}^{e,a+s}_{s,\mu_1,r}.
\end{equation}
\end{itemize}
\end{theorem}

\bk

The following result concerns the scenario where the members of one  subgroup of the population disperse very slowly while the members of the other subgroup disperse very fast.  

\begin{theorem}\label{Th2-5}\begin{itemize}\item[\rm (i)]Let $\eta^*_1$ be  given by 
\begin{equation}\label{Th2-5-Eq1}
    \eta_1^*=\inf\left\{ \eta\in (-(a+s)_{\min},\infty)\ :\ \int_{\Omega}\left(\frac{rs}{\eta+a+s}-(e+\eta)\right)<0\right\}.
\end{equation}
Then 
\begin{equation}\label{Th2-5-Eq2}
    \lim_{\mu_1\to 0,\mu_2\to\infty}\lambda_p(\mu\circ\mathcal{K}+\mathcal{A})=\eta_1^*.
\end{equation}
Moreover  $\eta_1^*$ and $\lim_{\eta\to0^+}\int_{\Omega}\left(\frac{rs}{\eta+a+s}-(\eta+e)\right)$ have the same sign. 

\item[\rm (ii)] Let $\eta^*_2$ be  given by 
\begin{equation}\label{Th2-5-Eq4}
    \eta_2^*=\inf\left\{ \eta\in (-e_{\min},\infty)\ :\ \int_{\Omega}\left(\frac{rs}{\eta+e}-(a+s+\eta)\right)<0\right\}.
\end{equation}
Then 
\begin{equation}\label{Th2-5-Eq5}
    \lim_{\mu_1\to \infty,\mu_2\to0}\lambda_p(\mu\circ\mathcal{K}+\mathcal{A})=\eta_2^*.
\end{equation}Moreover  $\eta_2^*$ and $\lim_{\eta\to0^+}\int_{\Omega}\left(\frac{rs}{\eta+e}-(\eta+a+s)\right)$ have the same sign.

 \end{itemize}
    
\end{theorem}

Thanks to the above results on the limit of $\lambda_p(\mu\circ\mathcal{K}+\mathcal{A})$ with respect to the dispersal rates $\mu$, we can now provide a complete information on the sign of $\lambda_p(\mu\circ\mathcal{K}+\mathcal{A})$ when $\mu$ is near either the $\mu_1$-axis or $\mu_2$-axis in the $\mu_1\times\mu_2$-plane.

\begin{theorem}\label{Th2-6}  Let $\eta_1^*$ be given by Theorem \ref{Th2-5}-{\rm (i)}.
\begin{itemize}
    \item[\rm (i)] If $\eta_1^*>0$, then there exist  $\delta_1^{*}>0$ and $m^*>0$ such that $\lambda_{p}(\mu\circ\mathcal{K}+\mathcal{A})\ge m^*$ for every $\mu=(\mu_1,\mu_2)\in(0,\delta_1^{*})\times\mathbb{R}_+$.

    \item[\rm (ii)] If $\Lambda_{\max}<0$, then there exist $\delta_1^{*}>0$ and $m^*<0$ such that  $\lambda_{p}(\mu\circ\mathcal{K}+\mathcal{A})<m^*$ for every $\mu=(\mu_1,\mu_2)\in(0,\delta_1^{*})\times\mathbb{R}_+$.

    \item[\rm (iii)] If $(a+s)_{\min}>0$ and $\eta_1^*<0<\Lambda_{\max}$, then  for every $0<\varepsilon\ll 1$ there exist $\delta_{1}^*>0$ and $m^*>0$ such that $\lambda_p(\mu\circ\mathcal{K}+\mathcal{A})>m^*$ for every $\mu=(\mu_1,\mu_2)\in(0,\delta_{1}^*)\times(0,\mu_2^*-\varepsilon)$ and $\lambda_p(\mu\circ\mathcal{K}+\mathcal{A})<-m^*$ for every $\mu=(\mu_1,\mu_2)\in(0,\delta_{1}^*)\times(\mu_2^*+\varepsilon,\infty),$  where $\mu_2^*>0$ is uniquely determined by $\lambda_{p}\Big(\mu_2^*\mathcal{K}+\Big(\frac{rs}{a+s}-e\Big)\mathcal{I}\Big)=0$.
\end{itemize}
    
\end{theorem}

\begin{remark}Theorem \ref{Th2-6} gives the sign of the principal spectrum point when the juvenile's dispersal rate is small. Similar result hold when the adult's dispersal rate is small. For simplicity, suppose that $ e_{\min} >0$ and $(a+s)_{\min}>0$. Following the theory of the basic reproduction number for population models or infectious disease models \cite{DH2000, DHM1990, VW2002}, it can be shown using the next generation matrix techniques that the quantity 
\begin{equation}\label{R-0}
   \mathcal{R}_0(x)=\frac{r(x)s(x)}{(a(x)+s(x))e(x)}\quad \forall\ x\in\overline{\Omega},
\end{equation}
is the {\it local basic reproduction}  function for system \eqref{model}. It tells us about the local persistence of the population in the sense that if $\mathcal{R}_0(x)>1$ at some location $x\in\Omega$, then the species can persist in such location in the absence of dispersal.  However, the species dies out if not dispersing in the locations where  $\mathcal{R}_0(x)<1$. It is clear from the $\mathcal{R}_0$ expression given in \eqref{R-0} that when $\min_{x\in\overline{\Omega}}\mathcal{R}_0(x)>1$, then $\int_{\Omega}\Big(\frac{rs}{a+s}-e\Big)>0$ and $ \int_{\Omega}\Big(\frac{rs}{e}-(a+s)\Big)>0$. In this case, it follows from Theorem \ref{Th2-6} that there is $\eta^*>0$ such for any choice of dispersal rate $\mu=(\mu_1,\mu_2)$ satisfying $\min\{\mu_1,\mu_2\}<\eta^*$, the species always persists.   As a result, if the species' local reproduction function is bigger than one, the population persists as long  as one of the subgroups moves slowly. 
\end{remark}

\subsubsection{Asymptotic profiles of steady states}

 For every $x\in\overline{\Omega}$, let ${\bf V}(x)=(V_1(x),V_2(x))$ denote  the unique nonnegative stable solution of the system of algebraic equations

\begin{equation}\label{kin}
    \begin{cases}
        0=r(x)V_1(x)-(a(x)+s(x)+b(x)V_1(x)+c(x)V_2(x))V_1(x) \cr 
        0=s(x)V_1(x)-(e(x)+f(x)V_2(x)+g(x)V_1(x))V_2(x).
    \end{cases}
\end{equation}

It follows from \cite[Theorem 2]{OSUU2023} that ${\bf V}(x)$ is positive if and only if $\Lambda(x)>0$. Thanks to \cite[Theorem 1]{OSUU2023}, system \eqref{model} has no positive steady-state solution if $\lambda_p(\mu\circ\mathcal{K}+\mathcal{A})\le 0$. Thus, by Theorem \ref{Th2-1}, to ensure that \eqref{model} has a positive steady state solution for small dispersal rates $\mu$, it is necessary to assume that $\Lambda_{\max}>0$. The following result demonstrates the asymptotic profiles of positive steady-state solutions of \eqref{model} for small dispersal rates $\mu$.

\begin{theorem}\label{TH6}
Suppose that $\Lambda_{\max}>0$. Then every positive steady state solution ${\bf u}^{\mu}$ of \eqref{model} for small dispersal rate $\mu$ satisfies ${\bf u}^{\mu}\to {\bf V} $ as $\max\{ \mu_1,\mu_2\}\to  0$, for $x$ uniformly in $\Omega$.
\end{theorem}

 Theorem \ref{TH6} shows that spatial distribution of the steady-state solutions of \eqref{model} for small dispersal rates of the population is completely determined by the kinetic system  \eqref{kin}.  Our next result concerns the asymptotic profiles of positive steady-state solutions of \eqref{model} for large dispersal rates, $\mu$.

\begin{theorem}\label{TH7}
Suppose that $\tilde{\Lambda}>0$. Then every positive steady state solution ${\bf u}^{\mu}$ of \eqref{model} for large dispersal rates $\mu$ satisfies ${\bf u}^{\mu}\to \tilde{\bf V} $ as $\min\{\mu_1,\mu_2\}\to {\infty}$, for $x$ uniformly in $\Omega$, where $\tilde{\bf V}$ is the unique positive solution of 
\begin{equation}\label{TH7-eq1}
    \begin{cases}
        0=\hat{r}\cdot\tilde{V}_1-(\hat{a}+\hat{s}+\hat{b}\cdot\tilde{V}_1+\hat{c}\cdot\tilde{V}_2)\tilde{V}_1\cr 
        0=\hat{s}\cdot\tilde{V}_1-(\hat{e}+\hat{f}\cdot\tilde{V}_2+\hat{g}\cdot\tilde{V}_1)\tilde{V}_2.
    \end{cases}
\end{equation}
    
\end{theorem}

 Theorem \ref{TH7} shows that the steady-state solutions of \eqref{model} are uniformly distributed on the whole habitat when  the dispersal rates of the population are significantly large. Moreover, their spatial distribution is determined by that of the kinetic model for which the parameters are replaced by their averages. The next result discusses the asymptotic profiles of positive steady states when $\mu_1$ approaches zero for every fixed $\mu_2$. For convenience, we introduce the function

\begin{equation}\label{H-eq}
    H(x,\tau)=\frac{1}{2b(x)}\Big(\sqrt{(a(x)+s(x)+c(x)\tau)^2+4b(x)r(x)\tau}-(a(x)+s(x)+c(x)\tau)\Big)\quad \forall\ \tau\ge 0, \ x\in\overline{\Omega}.
\end{equation}
The following result holds.

\begin{theorem}\label{TH8}  Fix $\mu_2>0$ and suppose that $(a+s)_{\min}>0$.  Suppose also that $\lambda_{\mu_2}^{a+s,e}>0$, where $\lambda_{\mu_2}^{a+s,e}$ is given by Theorem \ref{Th2-3}-{\rm (i)}. Then any positive steady-state solution ${\bf u}^{\mu}$ of \eqref{model} for small values of $\mu_1$ satisfies ${\bf u}^{\mu}\to (H(\cdot,w^*),w^*)$ as $\mu_1\to0$ uniformly in $\Omega$ where $H$ is defined by \eqref{H-eq} and $w^*\in X^{++}$ is the unique positive solution of 
\begin{equation}\label{TH8-eq}
        0=\mu_2\mathcal{K}w^*+sH(\cdot,w^*)-(e+fw^*+gH(\cdot,w^*))w^* \quad x\in\Omega.
\end{equation}
    
\end{theorem}

 Note that similar result will hold if we fix $\mu_1$ and let $\mu_2 \to 0$.  We complete our results with the asymptotic profiles of positive steady-state solutions of \eqref{model} as one of the diffusion get arbitrarily large. In this direction, the following result holds.

\begin{theorem}\label{TH9}
   Fix $\mu_2>0$ and suppose that $e\in X^+\setminus\{0\}$. Suppose also that $\tilde{\lambda}^{a+s,e}_{r,\mu_2,s} >0$, where $\tilde{\lambda}^{a+s,e}_{r,\mu_2,s} $ is given by Theorem \ref{Th2-4}-{\rm(i)}. Then any positive steady-state solution ${\bf u}^{\mu}$ of \eqref{model} for large values of $\mu_1$ satisfies, up to a subsequence, ${\bf u}^{\mu}\to (l^*,\tilde{w}^*)$ as $\mu_1\to\infty$ uniformly in $\Omega$ where $l^*>0$,  $\tilde{w}^*\in X^{++}$ and satisfy  
\begin{equation}\label{TH9-eq}
       \begin{cases} 
       0=\int_{\Omega}r\tilde{w}^* -\Big(\int_{\Omega}(a+s)+l^*\int_{\Omega}b+\int_{\Omega}c\tilde{w}^*\Big)l^* \cr
       0=\mu_2\mathcal{K}\tilde{w}^*+sl^*-(e+f\tilde{w}^*+gl^*)\tilde{w}^* & x\in\Omega.
       \end{cases}
\end{equation}
    
\end{theorem}
 Note that similar result will hold if we fix $\mu_1$ and let $\mu_2 \to \infty$. Theorems \ref{TH8} and \ref{TH9} provide information on the spatial distribution of the steady state solution of system \ref{model} when one of the dispersal rate is fixed while the other is either very small or very large. The limiting equation \eqref{TH8-eq} satisfied by $w^*$, is uniquely solvable  for a positive solution. We expect that such information could be used to obtain some partial results on the uniqueness and (linear) stability of steady-state solutions of system \eqref{model} when one dispersal rate is sufficiently small.   However, due to the nonlocal terms in \eqref{TH9-eq}, we  could not establish the uniqueness of its solution. This highlights some of the difficulties associated with the questions of uniqueness and stability of steady-state solutions of \eqref{model}. We hope that the results established in the current study will be of great help in tackling these open questions.

\section{Preliminaries}
\label{Sec3}

\noindent In the current section, we collect some preliminary results needed for establishing our main results in the subsequent sections.  Recall that the linearization of \eqref{model} at the trivial solution was given by \eqref{linera-model}.  Let $\{{\bf U}^{\mu}(t)\}_{t\ge 0}$ denote the uniformly continuous semigroup generated by the bounded linear operator $\mu\circ\mathcal{K}+\mathcal{A}$ on ${X}\times{X}$. Hence, for every ${\bf U}_0\in X\times X$, ${\bf U}(t,{\bf U}_0):={\bf U}^{\mu}(t){\bf U}_0$ is the unique solution of \eqref{M-eq3}, equivalently \eqref{linera-model}, with initial data ${\bf U}_0$. Since $A(x)$ is cooperative for each $x\in \overline{\Omega}$, the uniformly continuous semigroup $\{{\bf U}^{\mu}(t)\}_{t\ge 0}$  is strongly positive, in the sense that 
 \begin{equation}\label{M-eq4}
 {\bf U}^{\mu}(t)(X^+\times X^+)\subset X^+\times X^+\quad\text{and}\quad {\bf U}^{\mu}(t)((X^+\times X^+)\setminus\{{\bf 0}\})\subset X^{++}\times X^{++}\quad  \forall\ t>0.
 \end{equation}

\noindent The following result shows that the  principal spectrum point  $ \lambda_p(\mu\circ\mathcal{K}+\mathcal{A})$ equals the exponential growth bound of the uniformly continuous semigroup $\{{\bf U}^{\mu}(t)\}_{t\ge 0}$. This result turns out to be of particular importance in the arguments used to prove some  of our main results.

 \begin{proposition}\label{Prop-1} 
$\lambda_p(\mu\circ\mathcal{K}+\mathcal{A})\in \sigma(\mu\circ\mathcal{K}+\mathcal{A})$ and 
\begin{equation}\label{M-eq6} \lambda_p(\mu\circ\mathcal{K}+\mathcal{A})= \lim_{t\to\infty}\frac{\ln(\|{\bf U}^{\mu}(t)\|)}{t}.
     \end{equation}
 
 \end{proposition}
 \begin{proof} See \cite[Proposition A1]{OSUU2023}. 
 \end{proof}
 
 \noindent Let $\{e^{t\mathcal{K}}\}_{t\ge0}$ denote the uniformly continuous semigroup of bounded linear operators generated by the bounded linear operator $\mathcal{K}$ on $X$. 
 This means that given any $u_0\in X$, $u(t,x;u_0):=(e^{t\mathcal{K}}u_0)(x)$ is the unique solution of 
\begin{equation}\label{D-Eq2}
    \begin{cases}
    \partial_tu=\mathcal{K}u(t,x) & x\in\bar{\Omega},\cr
    u(0,\cdot)=u_0(x) & x\in\bar{\Omega}.
    \end{cases}
\end{equation}
By the comparison principle for linear nonlocal operators,    $\{e^{t\mathcal{K}}\}_{t\ge 0}$ is strongly monotone, in the sense that  
\begin{equation}\label{sa3}
e^{t\mathcal{K}}u<e^{t\mathcal{K}}v\quad  \text{for all}\ t>0\ \text{whenever } \ u\le\not\equiv v, \ u,v\in   X. 
\end{equation}
Let 
\begin{equation}\label{D-Eq8}
    \beta_*:=\inf_{u\in L^2(\Omega)\setminus\{0\},\ \int_{\Omega}u=0}\frac{\frac{1}{2}\int_{\Omega}\int_{\Omega}\kappa(x,y)(u(x)-u(y))^2dydx}{\int_{\Omega}u^2(x)dx}.
\end{equation}

\noindent It follows from \cite{CCR2006} that $\beta_*$ is a positive number. 
With respect to the asymptotic behavior of solutions of \eqref{D-Eq2}, the following result holds. 

\begin{lemma}[\cite{CCR2006}]\label{D-1-lemma}
For every $u_0\in L^2(\Omega)$,  it holds that 
\begin{equation}\label{D-Eq9}
    \int_{\Omega}e^{t\mathcal{K}}u_0=\int_{\Omega}u_0\quad \forall\ t>0,
\end{equation}
hence $\{e^{t\mathcal{K}}\}_{t\ge 0}$ leaves invariant  the Banach space $\{u\in L^2(\Omega) : \int_{\Omega}u=0\}$. Furthermore, 
\begin{equation}\label{D-Eq3}
    \Big\|e^{t\mathcal{K}}u-\frac{1}{|\Omega|}\int_{\Omega}u\Big\|_{L^2(\Omega)}\le e^{-t\beta_{*} }\Big\|u-\frac{1}{|\Omega|}\int_{\Omega}u\Big\|_{L^2(\Omega)}\quad \forall\ t>0,\ u\in L^2(\Omega),
\end{equation}
and 
\begin{equation} \label{D-Eq4}
    \Big\|u-\frac{1}{|\Omega|}\int_{\Omega}u\Big\|_{L^2(\Omega)}^2\le \frac{1}{2\beta_*}\int_{\Omega}\int_{\Omega}\kappa(x,y)(u(x)-u(y))^2dydx\quad  u\in L^2(\Omega).
\end{equation}
\end{lemma}

\noindent For convenience, we define 
\begin{equation}
    K(x)=\int_{\Omega}\kappa(x,y)dy \quad  x\in\overline{\Omega}.
\end{equation}
By hypothesis {\bf (H1)}, $K$ is H\"older continuous on $X$  and $K\in X^{++}$; hence, $K_{\min}>0.$ Next,  let $\xi>0$  and $h\in X$ be fixed. Let $\lambda_p(\xi\mathcal{K}-h\mathcal{I})$  denote the principal spectrum point  of the bounded linear operator $\xi\mathcal{K}-h\mathcal{I}$ on $X$. Thus, as in Proposition \ref{Prop-1}, $\lambda_p(\xi\mathcal{K}-h\mathcal{I})\in \sigma(\xi\mathcal{K}-h\mathcal{I})$ and 
\begin{equation}\label{Pl7-1}
\lambda_p(\xi\mathcal{K}-h\mathcal{I})=\lim_{t\to\infty}\frac{\ln\big(\big\|e^{t(\xi\mathcal{K}-h\mathcal{I})}\big\|\big)}{t},
\end{equation}
where $ \{e^{t(\xi\mathcal{K}-h\mathcal{I})}\}_{t\ge 0}$ denotes the strongly positive uniformly continuous  semigroup generated by the bounded linear operator $\xi\mathcal{K}-h\mathcal{I}$ on $X$. Note that for every $\nu>\lambda_p(\xi\mathcal{K}-h\mathcal{I})$, the bounded linear operator $\nu\mathcal{I}-(\xi\mathcal{K}-h\mathcal{I})$ is invertible.  Denote  by $\Psi_{\xi,h}(\cdot,\nu,\cdot)=(\nu\mathcal{I}-(\xi\mathcal{K}-h\mathcal{I}))^{-1}$ the inverse operator of $\nu\mathcal{I}-(\xi\mathcal{K}-h\mathcal{I})$, that is for every $w\in X$, $\Psi_{\xi,h}(x,\nu,w)$ is the unique solution of the equation 
\begin{equation}\label{pl2-1}
    \nu \Psi=\xi\mathcal{K}\Psi-h\Psi+w.
\end{equation}
Note that since \eqref{Pl7-1} holds (see also  \cite{Engel_Nagel}), then 
\begin{equation}\label{Pl6-2}
    \Psi_{\xi,h}(\cdot,\nu,w)=\int_{0}^{\infty}e^{-\nu t}e^{t(\xi\mathcal{K}-h\mathcal{I})}w dt\quad \forall\ w\in X,
\end{equation}
where the above integral is defined with respect to the uniform topology.

 \section{ Proofs  of Theorems \ref{TH4} and \ref{Th2-1}} \label{Sec4}
 This section is devoted for the proofs of Theorems \ref{TH4} and \ref{Th2-1}.
 
 \begin{proof}[Proof of Theorem \ref{TH4}] The proof is divided into two steps.
 
\noindent {\bf Step 1}. In this step, we show that 
 \begin{equation}\label{HK4}
    \lambda^*(\mu\circ\mathcal{K}+\mathcal{A})=\lambda_p(\mu\circ\mathcal{K}+\mathcal{A}).
 \end{equation}
 It is clear from the positivity of the uniformly continuous semigroup $\{{\bf U}^{\mu}(t)\}_{t\ge 0}$ that  $\lambda^*(\mu\circ\mathcal{K}+\mathcal{A})\ge \lambda_p(\mu\circ\mathcal{K}+\mathcal{A})$. Now, let $\varepsilon>0$ be fixed. Note from the definition of $\lambda_p(\mu\circ\mathcal{K}+\mathcal{A})$ that  $\lambda_p(\mu\circ\mathcal{K}+\mathcal{A})+\varepsilon$ belongs to the resolvent set of the bounded linear operator $\mu\circ\mathcal{K}+\mathcal{A}$. Furthermore, by \eqref{M-eq6}, it follows from the Hille--Yosida theorem that $ ((\lambda_p(\mu\circ\mathcal{K}+\mathcal{A})+\varepsilon)\mathcal{I}-(\mu\circ\mathcal{K}+\mathcal{A}))^{-1}$ is the Laplace transform of the positive  uniformly continuous  semigroup $\{{\bf U}^{\mu}(t)\}_{t\ge 0}$, where $((\lambda_p(\mu\circ\mathcal{K}+\mathcal{A})+\varepsilon)\mathcal{I}-(\mu\circ\mathcal{K}+\mathcal{A}))^{-1}$ denotes the inverse of the bounded linear operator $(\mu\circ\mathcal{K}+\mathcal{A})+\varepsilon)\mathcal{I}-(\mu\circ\mathcal{K}+\mathcal{A})$. Hence, for every $\psi\in X^{++}\times X^{++}$, we have that $ ((\mu\circ\mathcal{K}+\mathcal{A})+\varepsilon)\mathcal{I}-(\mu\circ\mathcal{K}+\mathcal{A}))^{-1}\psi\in X^{++}\times X^{++} $. In particular for $\psi={\bf 1}$,
 \begin{equation} \label{IL9}
 \varphi:=((\mu\circ\mathcal{K}+\mathcal{A})+\varepsilon)\mathcal{I}-(\mu\circ\mathcal{K}+\mathcal{A}))^{-1}{\bf 1}\in X^{++}\times X^{++}
\end{equation}
 and 
 \begin{equation} \label{IL10}
 (\lambda_p(\mu\circ\mathcal{K}+\mathcal{A})+\varepsilon)\varphi =(\mu\circ\mathcal{K}+\mathcal{A})\varphi+{\bf 1}>(\mu\circ\mathcal{K}+\mathcal{A})\varphi.
 \end{equation}
 Therefore, $\lambda_p(\mu\circ\mathcal{K}+\mathcal{A})+\varepsilon\ge \lambda^*(\mu\circ\mathcal{K}+\mathcal{A})$. Since $\varepsilon>0$ was arbitrarily chosen, then $\lambda_p(\mu\circ\mathcal{K}+\mathcal{A})\ge \lambda^*(\mu\circ\mathcal{K}+\mathcal{A})$, which completes the proof of Step 1.
 
 \noindent{\bf Step 2}. In the current step, we show that 
 \begin{equation}\label{HK5}
      \lambda_*(\mu\circ\mathcal{K}+\mathcal{A})=\lambda_p(\mu\circ\mathcal{K}+\mathcal{A}).
 \end{equation}
 To this end, let $\sigma:=\lambda_p(\mu\circ\mathcal{K}+\mathcal{A})-\lambda_*(\mu\circ\mathcal{K}+\mathcal{A})$. It is clear from \eqref{HK3} and \eqref{HK4} that $\sigma\ge 0$.  We shall show that $\sigma=0$. Suppose on the contrary that $\sigma>0$. Hence, \begin{equation} \label{HK6}
  \lambda_p(\mu\circ\mathcal{K}+\mathcal{A}-\lambda_*(\mu\circ\mathcal{K}+\mathcal{A}){\bf 1}\circ\mathcal{I})=\lambda_p(\mu\circ\mathcal{K}+\mathcal{A})-\lambda_*(\mu\circ\mathcal{K}+\mathcal{A})=\sigma>0.
 \end{equation}
 Therefore, it follows from \cite[Theorem 4-(i)]{OSUU2023} that when $\mathcal{A}$ is replaced with $\mathcal{A}-\lambda_*(\mu\circ\mathcal{K}+\mathcal{A}){\bf 1}\circ\mathcal{I}$ in system \eqref{model} with $c=g\equiv 0$,  there is a positive steady-state solution ${\bf u}^{**}$ of the system 
 \begin{equation} \label{HK7}
   \begin{cases}
      0=\mu_1\mathcal{K}u_1^{**} +ru_2^{**}-(s+a+\lambda_*(\mu\circ\mathcal{K}+\mathcal{A}) +bu_1^{**})u_1^{**} & x\in\overline{\Omega},\cr
      0=\mu_2\mathcal{K}u_2^{**} +su_1^{**}-(e+\lambda_*(\mu\circ\mathcal{K}+\mathcal{A})+fu_2^{**})u_2^{**} & x\in\overline{\Omega}.
     \end{cases}
 \end{equation}
 Hence, with $\varepsilon^*:= \min_{x\in\overline{\Omega}}\min \{b(x)u_1^{**}(x),f(x)u_2^{**}(x)\}>0$, it follows from \eqref{HK7} that 
 $$ 
 (\varepsilon^*+\lambda_*(\mu\circ\mathcal{K}+\mathcal{A})){\bf u}^{**}\leq (\mu\circ\mathcal{K}+\mathcal{A}){\bf u}^{**},
 $$
 which implies that $\varepsilon^*+ \lambda_*(\mu\circ\mathcal{K}+\mathcal{A})\le \lambda_*(\mu\circ\mathcal{K}+\mathcal{A})$. This is impossible since $\varepsilon^*>0$. Therefore, $\sigma=0$, hence \eqref{HK5} holds. This completes the proof of the Theorem.
 
 \end{proof}
 
\noindent Now, for every $x\in\overline{\Omega}$, let $\Lambda(x)$ denote the maximal eigenvalue of the matrix $A(x)$ and ${\bf Q}(x)$ be an associated nonnegative eigenvector satisfying $Q_1(x)+Q_2(x)=1$. By direct computations,
 \begin{equation}\label{HK8}
     \Lambda(x)=\frac{1}{2}\Big( \sqrt{(a(x)+s(x)-e(x))^2+4r(x)s(x))}-(a(x)+s(x)+e(x))\Big)\quad \forall\ x\in\overline{\Omega},
 \end{equation}
 and $Q_1(x)$ and $Q_2(x)$ satisfy
 \begin{eqnarray}\label{HK9}
 (\Lambda(x)+a(x)+s(x))Q_1(x)=r(x)Q_2(x) \quad \forall\ x\in\overline{\Omega}.
 \end{eqnarray}
 Hence, since $Q_1(x)=1-Q_2(x)$,
 \begin{equation}\label{HK10} \Lambda(x)+a(x)+s(x)=(r(x)+\Lambda(x)+a(x)+s(x))Q_2(x).
 \end{equation}
 However,
 \begin{align}\label{HK11}
    \Lambda(x)+a(x)+s(x)=\frac{1}{2}\Big( \sqrt{(a(x)+s(x)-e(x))^2+4r(x)s(x)}+(a(x)+s(x)-e(x))\Big).
 \end{align}
 Hence, if  $r(x)s(x)>0$, we have that $r(x)+a(x)+s(x)>0$, in which case it follows from \eqref{HK10} that 
 \begin{equation}\label{HK12}
     Q_2(x)=\frac{\Lambda(x)+a(x)+s(x)}{r(x)+\Lambda(x)+a(x)+s(x)}>0
 \end{equation}
 and 
 \begin{equation}\label{HK13}
     Q_1(x)=\frac{r(x)}{r(x)+ \Lambda(x) +a(x)+s(x)}>0.
 \end{equation}
 In view of \eqref{HK12} and \eqref{HK13}, if $r,s\in X^{++}$, then ${\bf Q}\in X^{++}\times X^{++}$.

 \begin{lemma}\label{Lem-3}
     Let $\Lambda(x)$ be the maximal eigenvalue of the cooperative matrix $A(x)$, then,
     \begin{equation}\label{HK15}
    \limsup_{\max\{\mu_1,\mu_2\}\to { 0}} \lambda_p(\mu\circ\mathcal{K}+\mathcal{A})\leq  \Lambda_{\max}.
 \end{equation}
 \end{lemma}

 \begin{proof} Let $\varepsilon>0$ be fixed and set $r^{\varepsilon}=r+\varepsilon$ and $s^{\varepsilon}=s+\varepsilon$. Next, let ${\bf Q}^{\varepsilon}(x)$ be given by \eqref{HK12} and $\eqref{HK13}$ where $r$ and $s$ are replaced with $r^{\varepsilon}$ and $s^{\varepsilon}$. Then, there is $\sigma_0>0$ such that 
 \begin{equation}\label{HK14}
      (\Lambda_{\max}^{\varepsilon}+2\varepsilon){\bf Q}^{\varepsilon}\ge (\mu\circ\mathcal{K}+\mathcal{A}){\bf Q}^{\varepsilon}\quad \forall\ 0<\mu_1,\mu_2\le \sigma_0,
 \end{equation}
 where $\Lambda^{\varepsilon}(x)$ is given by \eqref{HK8} with $r$ and $s$ replaced with $r^{\varepsilon}$ and $s^{\varepsilon}$.  To see this, observe that from \eqref{HK12} and \eqref{HK13}, we know that ${\bf Q}^{\varepsilon}\in X^{++}\times X^{++}.$
  Hence,  by computations
 \begin{align*}
    & \mu_1\int_{\Omega}\kappa(x,y)(Q_1^{\varepsilon}(y)-Q_1^{\varepsilon}(x))dy+r(x)Q_2^{\varepsilon}(x)-(a(x)+s(x))Q_1^{\varepsilon}(x)\cr 
    =& \mu_1\int_{\Omega}\kappa(x,y)(Q_1^{\varepsilon}(y)-Q_1^{\varepsilon}(x))dy+\Big(r^{\varepsilon}(x)Q_2^{\varepsilon}(x)-(a(x)+s^{\varepsilon}(x))Q_1^{\varepsilon}(x)\Big) +\varepsilon(Q_1^{\varepsilon}(x)-Q_2^{\varepsilon}(x))\cr 
    =& \mu_1\int_{\Omega}\kappa(x,y)(Q_1^{\varepsilon}(y)-Q_1^{\varepsilon}(x))dy+\Lambda^{\varepsilon}(x)Q_1^{\varepsilon}(x) +\varepsilon(Q_1^{\varepsilon}(x)-Q_2^{\varepsilon}(x))\cr 
    \le & \mu_1\|K\|_{\infty}\|Q_1^{\varepsilon}\|_{\infty}|\Omega| +(\Lambda_{\max}^{\varepsilon}+\varepsilon)Q_1^{\varepsilon}(x)\cr 
    \le& (\Lambda_{\max}^{\varepsilon}+2\varepsilon)Q_1^{\varepsilon}(x)
 \end{align*}
 whenever $\mu_1\le \frac{\varepsilon Q_{1,\min}^{\varepsilon}}{\|K\|_{\infty}\|Q_1^{\varepsilon}\|_{\infty}|\Omega|}$. Similar computations show that 
 $$ 
 \mu_2\int_{\Omega}\kappa(x,y)(Q_2^{\varepsilon}(y)-Q_2^{\varepsilon}(x))dy+s(x)Q_1^{\varepsilon}(x)-e(x)Q_2^{\varepsilon}(x)\le (\Lambda_{\max}^{\varepsilon}+2\varepsilon)Q_2^{\varepsilon}(x)
 $$
 whenever $\mu_2\le \frac{\varepsilon Q_{2,\min}^{\varepsilon}}{\|K\|_{\infty}\|Q_2^{\varepsilon}\|_{\infty}|\Omega|}$. Whence, taking $\sigma_0:=\frac{\varepsilon }{\|K\|_{\infty}|\Omega|}\min\Big\{\frac{ Q_{1,\min}^{\varepsilon}}{\|Q_1^{\varepsilon}\|_{\infty}},\frac{ Q_{2,\min}^{\varepsilon}}{\|Q_2^{\varepsilon}\|_{\infty}}\Big\}$, then \eqref{HK14} follows. Now, from \eqref{HK14}, since ${\bf Q}^{\varepsilon}\in X^{++}\times X^{++}$, then  by Theorem \ref{TH4}, 
 $$
 \Lambda_{\max}^{\varepsilon}+2\varepsilon\ge  \lambda^*(\mu\circ\mathcal{K}+\mathcal{A})=\lambda_p(\mu\circ\mathcal{K}+\mathcal{A})\quad \forall\  0<\mu_1,\mu_2<\sigma_0.
 $$
 Since $\varepsilon$ is arbitrarily chosen and  $\lim_{\varepsilon\to 0^+}(\Lambda_{\max}^{\varepsilon}+2\varepsilon)=\Lambda_{\max}$, \eqref{HK15} then follows.
 \end{proof}

 \begin{lemma}\label{Lem-4}
      Let  $\Lambda(x)$  be the maximal eigenvalue of the cooperative matrix $A(x)$, then,
       \begin{equation}\label{HK19}
 \Lambda_{\max} \le \liminf_{\max\{\mu_1,\mu_2\}\to{ 0}}\lambda_p(\mu\circ\mathcal{K}+\mathcal{A}).
 \end{equation}
 \end{lemma}
 
 \begin{proof} 
 
  Let $\varepsilon_0>0$ be fixed.  Then there is  $x_0\in \Omega$ and $r_0>0$ such that 
 \begin{equation}\label{HK16}
 B(x_0,4r_0)\subset\Omega\quad \text{and}\quad \Lambda(x)\ge \Lambda_{\max}-\varepsilon_0 \quad \forall\ x\in B(x_0,4r_0).
 \end{equation}
 Now, choose $\psi\in C^{\infty}(\mathbb{R}^n)$ satisfying 
 \begin{equation}\label{HK17}
 0\le \psi\le 1,\quad \psi\equiv 1 \ \text{on}\ B(x_0.r_0)\quad \text{and}\quad \psi\equiv 0 \ \text{on}\ \mathbb{R}^n\setminus B(x_0,2r_0).
 \end{equation}
 We claim that
 \begin{equation}\label{HK18}
      (\Lambda_{\max}-\varepsilon_0-(\mu_1+\mu_2)\|K\|_{\infty})\psi{\bf Q}\le (\mu\circ\mathcal{K}+\mathcal{A})\psi{\bf Q}\quad \forall\ \mu=(\mu_1,\ \mu_2): \mu_1, \mu_2 > 0.
 \end{equation}
 \bk
 Indeed, by computations,
 \begin{align*}
      & \mu_1\int_{\Omega}\kappa(x,y)(\psi(y)Q_1(y)-Q_1(x)\psi(x))dy+r(x)Q_2(x)\psi(x)-(a(x)+s(x))Q_1(x)\psi(x)\cr
      \ge & - \mu_1Q_1(x)\psi(x)\|K\|_{\infty}+(r(x)Q_2(x)-(a(x)+s(x))Q_1(x))\psi(x)\cr
      =&- \mu_1Q_1(x)\psi(x)\|K\|_{\infty}+\Lambda(x)Q_1(x)\psi(x)\cr 
      \ge& - \mu_1Q_1(x)\psi(x)\|K\|_{\infty}+(\Lambda_{\max}-\varepsilon_0)Q_1(x)\psi(x)
 \end{align*}
 since \eqref{HK16} and \eqref{HK17} hold. Hence 
 \begin{align*}
 &\mu_1\int_{\Omega}\kappa(x,y)(\psi(y)Q_1(y)-Q_1(x)\psi(x))dy+r(x)Q_2(x)\psi(x)-(a(x)+s(x))Q_1(x)\psi(x)\cr
      \ge & (\Lambda_{\max}-\varepsilon_0- (\mu_1+\mu_2)\|K\|_{\infty})Q_1(x)\psi(x)\quad \forall \ x\in\overline{\Omega}.
      \end{align*}
      Similar arguments show that 
  \begin{align*}
 &\mu_2\int_{\Omega}\kappa(x,y)(\psi(y)Q_2(y)-Q_2(x)\psi(x))dy+s(x)Q_1(x)\psi(x)-e(x)Q_2(x)\psi(x)\cr
      \ge & (\Lambda_{\max}-\varepsilon_0- (\mu_1+\mu_2)\|K\|_{\infty})Q_1(x)\psi(x)\quad \forall \ x\in\overline{\Omega},
      \end{align*}
      hence \eqref{HK18} holds. Now, since $\{{\bf U}^{\mu}(t)\}_{t\ge 0}$ is positive and $\frac{\ln(\|{\bf U}^{\mu}(t)\|)}{t}\to\lambda_p(\mu\circ\mathcal{K}+\mathcal{A})$ as $t\to\infty$ (Proposition \ref{Prop-1}), it follows from \eqref{HK18} that 
$$ 
\Lambda_{\max}-\varepsilon_0-(\mu_1+\mu_2)\|K\|_{\infty}\le \lambda_p(\mu\circ\mathcal{K}+\mathcal{A}) \quad \forall\ \varepsilon_0,\ \mu_1,\ \mu_2>0.
$$
This implies that \eqref{HK19} holds since $\varepsilon_0$ is arbitrarily chosen. 
 \end{proof}

\noindent For every $\mu=(\mu_1,\mu_2)\in\mathbb{R}^{+}\times\mathbb{R}^+$, let $\Lambda_{\mu}(x)$ denote the maximal eigenvalue of the cooperative matrix 
\begin{equation}
    A_{\mu}(x):=A(x)-K(x)\text{Diag}(\mu)
\end{equation} 
where $\text{Diag}(\mu)$ is the $2\times 2$ diagonal matrix with diagonal entries $\mu_1$ and $\mu_2$. Hence, 
\begin{equation}\label{HK20}
    \Lambda_{\mu}(x)\le \Lambda(x)-\min\{\mu_1,\mu_2\}K_{\min}\leq \Lambda_{\max}-\min\{\mu_1,\mu_2\}K_{\min}.
\end{equation}

\begin{lemma}\label{Lem-5}
For every $\mu$ satisfying $\min\{\mu_1,\mu_2\}> \frac{\Lambda_{\max}+\max\{\|e\|_{\infty},\|a+s\|_{\infty}\}}{K_{\min}}$,  $\lambda_p(\mu\circ\mathcal{K}+\mathcal{A}) $ is the principal eigenvalue of the bounded linear operator $\mu\circ\mathcal{K}+\mathcal{A}$.

\end{lemma}
\begin{proof}For every $\mu$ satisfying $ \frac{\Lambda_{\max}+\max\{\|e\|_{\infty},\|a+s\|_{\infty}\}}{K_{\min}}<\min\{\mu_1,\mu_2\}$, it follows from inequalities \eqref{HK3} and \eqref{HK20} and Theorem \ref{TH4} that $\lambda_p(\mu\circ\mathcal{K}+\mathcal{A})>\max_{x\in\overline{\Omega}}\Lambda_{\mu}(x)$,  hence  the result follows from \cite[Theorem 2.1]{BS}.
\end{proof}

 \begin{proof}[Proof of Theorem \ref{Th2-1} ] {\rm (i)} It follows from Lemmas \ref{Lem-3} and \ref{Lem-4}.
 
 {\rm (ii)} 
 
 For every $\mu=(\mu_1,\mu_2)$ satisfying  $\min\{\mu_1,\mu_2\}> \frac{\Lambda_{\max}+\max\{\|e\|_{\infty},\|a+s\|_{\infty}\}}{K_{\min}}$,  by Lemma \ref{Lem-5}, $\lambda_p(\mu\circ\mathcal{K}+\mathcal{A})$ is the principal eigenvalue of the bounded linear operator $\mu\circ\mathcal{K}+\mathcal{A}$; so we can choose an associated strictly positive eigenfunction $\varphi^{\mu}$ satisfying 
 \begin{equation}\label{HK23}
     \|\varphi_1^{\mu}\|_{L^2(\Omega)}+\|\varphi_{2}^{\mu}\|_{L^2(\Omega)}=1. 
 \end{equation}
 From this point, we set $w_i^{\mu}=\varphi_i^{\mu}-\frac{1}{|\Omega|}\int_{\Omega}\varphi_i^{\mu}$, $i=1,2$. The rest of the proof is divided into two steps.
 
 \noindent {\bf Step 1.} In the current step, we shall show that 
 \begin{equation}\label{HK24}
     \lim_{\mu\to\infty}\sum_{i=1}^{2}\| w_i^{\mu}\bk\|_{L^2(\Omega)}=0.
 \end{equation}
 To this end, we first note from Lemma \ref{D-1-lemma} that 
 \begin{equation*}
      2\beta_*\|w_i^{\mu}\|_{L^2(\Omega)}^2\le\int_{\Omega}\int_{\Omega}\kappa(x,y)(\varphi_i^{\mu}(y)-\varphi_i^{\mu}(x))^2dydx, \quad i=1,2.
 \end{equation*}
 Therefore,
 \begin{align*}  \lambda_p(\mu\circ\mathcal{K}+\mathcal{A})\int_{\Omega}(\varphi_1^{\mu})^2=&\mu_1\int_{\Omega}\int_{\Omega}\kappa(x,y)(\varphi_1^{\mu}(y)-\varphi_1^{\mu}(x))\varphi_1^{\mu}(x)dydx+\int_{\Omega}r\varphi_1^{\mu}\varphi_2^{\mu}-\int_{\Omega}(a+s)(\varphi_{1}^{\mu})^2\cr 
   =&-\frac{\mu_1}{2}\int_{\Omega}\int_{\Omega}\kappa(x,y)(\varphi_1^{\mu}(y)-\varphi_1^{\mu}(x))^2dydx+\int_{\Omega}r\varphi_1^{\mu}\varphi_2^{\mu}-\int_{\Omega}(a+s)(\varphi_{1}^{\mu})^2\cr 
   \le&-\mu_1\beta_* \|w_1^{\mu}\|_{L^2(\Omega)}^2+\int_{\Omega}r\varphi_1^{\mu}\varphi_2^{\mu}-\int_{\Omega}(a+s)(\varphi_{1}^{\mu})^2.
 \end{align*}
 Hence, in view of \eqref{HK3} and \eqref{HK23} and the fact that $\varphi_i^{\mu}>0$ for each $i=1,2$, we obtain that 
 
 \begin{align}\label{HK24-2-1}
  \Big\|w_1^{\mu}\|_{L^2(\Omega)}^2\le & \frac{1}{\mu_1\beta_*}\left(\int_{\Omega}r\varphi_1^{\mu}\varphi_2^{\mu} -\int_{\Omega}(\lambda_p(\mu\circ\mathcal{K}+\mathcal{A})+(a+s))(\varphi_1^{\mu})^2\right)\cr 
  \leq & \frac{1}{\mu_1\beta_*}\left(\frac{\|r\|_{\infty}}{2}\|\varphi_2^{\mu}\|_{L^2(\Omega)}^2 +\frac{\|r\|_{\infty}-2\lambda_p(\mu\circ\mathcal{K}+\mathcal{A})}{2}\|\varphi_1^{\mu}\|_{L^2(\Omega)}^2\right)\cr 
  \le &\frac{1}{\mu_1\beta_{*}}\left(\frac{\|r\|_{\infty}}{2}\|\varphi_2^{\mu}\|_{L^2(\Omega)}^2+\frac{\|r\|_{\infty}+2\max\{\|a+s\|_{\infty},\|e\|_{\infty}\}}{2}\|\varphi_1^{\mu}\|_{L^2(\Omega)}^2\right)\cr 
  \le & \frac{1}{\mu_1\beta_{*}}\left(\frac{\|r\|_{\infty}}{2}+\frac{\|r\|_{\infty}+2\max\{\|a+s\|_{\infty},\|e\|_{\infty}\}}{2}\right)\cr
  =& \frac{(\|r\|_{\infty}+\max\{\|a+s\|_{\infty},\|e\|_{\infty}\})}{\mu_1\beta_{*}}.
 \end{align}
 Similar arguments show that 
\begin{equation}\label{HK24-2-2} 
  \Big\|w_2^{\mu}\|_{L^2(\Omega)}^2\leq \frac{(\|s\|_{\infty}+\max\{\|a+s\|_{\infty},\| e\|_{\infty}\})}{\mu_2\beta_{*}}.
 \end{equation}
 Therefore \eqref{HK24} holds from the last two inequalities.
 
\noindent {\bf Step 2.} In the current step, we complete the proof of the theorem. Observe that 
 \begin{equation}\label{HK26}
     \begin{cases}
     \lambda_p(\mu\circ\mathcal{K}+\mathcal{A})\left[\frac{1}{|\Omega|}\int_{\Omega}\varphi_1^{\mu}\right]=\frac{1}{|\Omega|}\int_{\Omega}rw_2^{\mu}-\frac{1}{|\Omega|}\int_{\Omega}(a+s)w_1^{\mu}+\frac{\hat{r}}{|\Omega|}\int_{\Omega}\varphi^{\mu}_2-\frac{(\hat{a}+\hat{s})}{|\Omega|}\int_{\Omega}\varphi^{\mu}_1\cr 
     \lambda_p(\mu\circ\mathcal{K}+\mathcal{A})\left[\frac{1}{|\Omega|}\int_{\Omega}\varphi_2^{\mu}\right]=\frac{1}{|\Omega|}\int_{\Omega}sw_1^{\mu}-\frac{1}{|\Omega|}\int_{\Omega}ew_2^{\mu}+\frac{\hat{s}}{|\Omega|}\int_{\Omega}\varphi^{\mu}_1-\frac{\hat{e}}{|\Omega|}\int_{\Omega}\varphi^{\mu}_2.
     \end{cases}
 \end{equation}
 Since $\lambda_p(\mu\circ\mathcal{K}+\mathcal{A})$ is bounded (see \eqref{HK3}), without loss of generality after passing to a subsequence, we may suppose that $\lambda_p(\mu\circ\mathcal{K}+\mathcal{A})\to \tilde{\lambda} $ as $\mu\to\infty$ for some real mumber $\tilde{\lambda}$. Since $\{\Big(\frac{1}{|\Omega|}\int_{\Omega}\varphi_1^{\mu},\frac{1}{|\Omega|}\int_{\Omega}\varphi_2^{\mu}\Big)\}_{\mu}$ is also nonnegative and bounded (see \eqref{HK23}), we may also suppose that there is some nonnegative real vector ${\bf Q}^*$  such that $\Big(\frac{1}{|\Omega|}\int_{\Omega}\varphi_1^{\mu},\frac{1}{|\Omega|}\int_{\Omega}\varphi_2^{\mu}\Big)\to {\bf Q}^*$ as $\mu\to\infty$. Hence, by \eqref{HK23} and \eqref{HK24}, we have that $Q_1^*+Q_2^{*}=\frac{1}{\sqrt{|\Omega|}}$ and $\varphi^{\mu}\to {\bf Q}^*$ as $\mu\to\infty$, in $L^2(\Omega)$. As a result, by letting $\mu\to\infty$ in \eqref{HK26} and recalling \eqref{HK24}, we obtain
 $$ 
 \tilde{\lambda}{\bf Q}^*=\hat{A}{\bf Q}^*,
 $$
 where the matrix $\hat{A}$ is given by \eqref{A-hat}. Finally, since ${\bf Q}^*$ is nonnegative, \bk $Q_1^{*}+Q_2^{*}=\frac{1}{\sqrt{|\Omega|}}>0$ and the matrix $\hat{A}$ is cooperative and irreducible, it then follows from Perron-Frobenius theorem that $\tilde{\lambda}=$ $\tilde{\Lambda}$ is the maximal eigenvalue of the matrix $\hat{A}$. Finally, since $\tilde{\lambda}=$ $\tilde{\Lambda}$ is independent of the subsequence of $\mu$ that we chose, we conclude that $\lambda_p(\mu\circ\mathcal{K}+\mathcal{A})\to$ $\tilde{\Lambda}$ as $\mu\to\infty$.
 \end{proof}

 \section{Proofs of Proposition \ref{Main-prop1} and Theorem \ref{Th2-3} }\label{Sec5}

\noindent Next, we discuss the proofs of  Proposition \ref{Main-prop1} and Theorem \ref{Th2-3}. To this end, we first present the proof of Propostion \ref{Main-prop1}.

 \begin{proof}[Proof of Proposition \ref{Main-prop1}] Let $\xi>0$ and $h,l\in X^{+}$ be fixed.
   Note that the function
  \begin{equation}\label{L-equation}
  L(\nu):= \lambda_p\Big(\xi\mathcal{K}+\big(\frac{rs}{\nu+h}-l\big)\mathcal{I}\Big)-\nu\quad \quad  \ \nu>-h_{\min},
  \end{equation}
   is continuous, strictly decreasing and 
   $$ 
   \lim_{\nu\to\infty}L(\nu)=-\infty.
   $$
   Therefore, the algebraic equation $L(\nu)=0$ has a unique solution, say $\lambda^{h,l}_{\xi}$ if and only if 
   $$ 
   \lim_{\nu\to-h_{\min}^{+}}L(\nu)=\sup_{\nu>-h_{\min}}L(\nu)>0,
   $$
   which proves the first part of the proposition. Moreover,  when $h_{\min} > 0$, if $L(0)<0$ then $\lambda_{\xi}^{h,l}<0$; if $L(0)>0$ then $\lambda_{\xi}^{h,l}>0$; and if $L(0)=0$ then $\lambda_{\xi}^{h,l}=0$. Therefore, since $L(0)=\lambda(\xi\mathcal{K}+\frac{rs}{h}-l)$, then $\lambda_{\xi}^{h,l}$ and $\lambda(\xi\mathcal{K}+\frac{rs}{h}-l)$ have the same sign. 
 \end{proof}

\noindent  For every $\eta\in X$ and $\xi>0$, it follows as in  Theorem \eqref{TH4} (see also \cite{ BCV}) that the following sup-inf characterization of the principal spectrum point, $\lambda_p(\xi\mathcal{K}+\eta\mathcal{I})$,  
\begin{equation}\label{sup-inf-char}
\lambda_p\Big(\xi\mathcal{K}+\eta \mathcal{I}\Big)=\lambda_*\Big(\xi\mathcal{K}+\eta\mathcal{I}\Big)=\lambda^*\Big(\xi\mathcal{K}+\eta\mathcal{I}\Big),
\end{equation}
holds, where
  \begin{equation*}
    \lambda_*\Big(\xi\mathcal{K}+\eta\mathcal{I}\Big):=\sup\Big\{\lambda\in\mathbb{R} : \exists \psi\in X^{++}\ \text{satisfying}\ \lambda\psi\le \xi\mathcal{K}\psi +\eta\psi \Big\} 
  \end{equation*}
  and 
 \begin{equation*}
   \lambda^*\Big(\xi\mathcal{K}+\eta \mathcal{I}\Big):=\inf\Big\{\lambda\in\mathbb{R} : \exists \psi\in X^{++}\ \text{satisfying}\ \lambda\psi\ge \xi\mathcal{K}\psi +\eta\psi \Big\}. 
  \end{equation*}
 In particular, for every $\nu>-(a+s)_{\min}$, it follows as in  Theorem \eqref{TH4} that 
  \begin{equation}\label{WE-4}
      \lambda_p\Big(\mu_2\mathcal{K}+(\frac{rs}{\nu+a+s}-e)\mathcal{I}\Big)=\lambda_*\Big(\mu_2\mathcal{K}+(\frac{rs}{\nu+a+s}-e)\mathcal{I}\Big)=\lambda^*\Big(\mu_2\mathcal{K}+(\frac{rs}{\nu+a+s}-e)\mathcal{I}\Big),
  \end{equation}

  \begin{proof}[Proof of Theorem \ref{Th2-3} ] {\rm (i)} Let $\mu_2>0$ be fixed. First, suppose that $\lambda_{\mu_2}^{a+s,e}>-(a+s)_{\min}$. Therefore,
  \begin{equation*}
      \lambda_p\Big(\mu_2\mathcal{K}+\Big(\frac{rs}{\lambda_{\mu_2}^{a+s,e}+a+s}-e \Big)\mathcal{I}\Big)=\lambda_{\mu_2}^{a+s,e}.
  \end{equation*}
  Hence, by \eqref{WE-4} for every $\varepsilon>0$, there exist $\varphi_{2,\varepsilon}^{+}\in X^{++}$ and $\varphi_{2,\varepsilon}^{-}\in X^{++}$ such that 
  \begin{equation}\label{WE-5}
   (\lambda_{\mu_2}^{a+s,e}+\varepsilon)\varphi_{2,\varepsilon}^{+}\ge \mu_2\mathcal{K}\varphi_{2,\varepsilon}^{+}+\frac{rs}{\lambda_{\mu_2}^{a+s,e}+a+s}\varphi^{+}_{2,\varepsilon}-e\varphi_{2,\varepsilon}^{+}
  \end{equation}
  and 
   \begin{equation}\label{WE-6}
      (\lambda_{\mu_2}^{a+s,e}-\varepsilon)\varphi_{2,\varepsilon}^{-}\le \mu_2\mathcal{K}\varphi_{2,\varepsilon}^{-}+\frac{rs}{\lambda_{\mu_2}^{a+s,e}+a+s}\varphi^{-}_{2,\varepsilon}-e\varphi_{2,\varepsilon}^{-}.
  \end{equation}
  First, let $\varphi_{1,\varepsilon}^{+}:=\frac{r+\varepsilon}{\lambda_{\mu_2}^{a+s,e}+a+s}\varphi_{2,\varepsilon}^+$. Then, it follows from \eqref{WE-5} that 
  \begin{equation*}
      \Big(\lambda_{\mu_2}^{a+s,e}+\varepsilon\Big(1+\frac{s}{\lambda_{\mu_2}^{a+s,e}+a+s} \Big)\Big)\varphi_{2,\varepsilon}^{+} \ge \mu_2\mathcal{K}\varphi_{2,\varepsilon}^{+}+s\varphi^{+}_{1,\varepsilon}-e\varphi_{2,\varepsilon}^{+}
  \end{equation*}
  and 
  \begin{align*}
    \lambda_{\mu_2}^{a+s,e}\varphi_{1,\varepsilon}^{+}=& (r+\varepsilon)\varphi_{2,\varepsilon}^{+}-(a+s)\varphi_{1,\varepsilon}^{+}\cr 
    =& \mu_1\mathcal{K}\varphi_{1,\varepsilon}^{+}+r\varphi_{2,\varepsilon}^{+}-(a+s)\varphi_{1,\varepsilon}^{+}+  \Big(\varepsilon-\mu_1\frac{\mathcal{K}\varphi_{1,\varepsilon}^{+}}{\varphi_{2,\varepsilon}^{+}}\Big)\varphi_{2,\varepsilon}^{+}  \cr 
    \ge & \mu_1\mathcal{K}\varphi_{1,\varepsilon}^{+}+r\varphi_{2,\varepsilon}^{+}-(a+s)\varphi_{1,\varepsilon}^{+}
  \end{align*}
  whenever  $\mu_1<\frac{\varepsilon}{1+\Big\| \frac{\mathcal{K}\varphi_{1,\varepsilon}^{+}}{\varphi_{2,\varepsilon}^{+}}\Big\|_{\infty}}$. Therefore, by Theorem \ref{TH4}, we have that 
  \begin{equation}\label{WE-7}
      \lambda_p(\mu\circ\mathcal{K}+\mathcal{A})\le \lambda_{\mu_2}^{a+s,e}+\varepsilon\Big\|1+\frac{s}{\lambda_{\mu_2}^{a+s,e}+a+s}\Big\|_{\infty} \quad \text{whenever}\quad \mu_1<\frac{\varepsilon}{1+\Big\| \frac{\mathcal{K}\varphi_{1,\varepsilon}^{+}}{\varphi_{2,\varepsilon}^{+}}\Big\|_{\infty}}.
  \end{equation}
      Next, let $\varphi_{1,\varepsilon}^{-}:=\frac{r}{\lambda_{\mu_2}^{a+s,e}+a+s}\varphi_{2,\varepsilon}^{-}$. Then, by \eqref{WE-6}, 
      \begin{equation*}
          (\lambda_{\mu_2}^{a+s,e}-\varepsilon)\varphi_{2,\varepsilon}^{-}\leq \mu_2\mathcal{K}\varphi_{2,\varepsilon}^{-}+s\varphi^{-}_{1,\varepsilon}-e\varphi_{2,\varepsilon}^{-}
      \end{equation*}
      and 
      \begin{align*}
(\lambda_{\mu_2}^{a+s,e}-\varepsilon)\varphi_{1,\varepsilon}^{-}=&r\varphi_{2,\varepsilon}^{-} -(a+s)\varphi_{1,\varepsilon}^{-}-\varepsilon\varphi_{1,\varepsilon}^{-}\cr 
=&\mu_1\mathcal{K}\varphi_{1,\varepsilon}^{-}+r\varphi_{2,\varepsilon}^{-} -(a+s)\varphi_{1,\varepsilon}^{-} -(\varepsilon+\mu_1\mathcal{K})\varphi_{1,\varepsilon}^{-}\cr 
=&\mu_1\mathcal{K}\varphi_{1,\varepsilon}^{-}+r\varphi_{2,\varepsilon}^{-} -(a+s)\varphi_{1,\varepsilon}^{-} -(\varepsilon-\mu_1K(\cdot))\varphi_{1,\varepsilon}^{-}-\mu_1\int_{\Omega}\kappa(\cdot,y)\varphi_{1,\varepsilon}^{-}dy\cr 
\le &\mu_1\mathcal{K}\varphi_{1,\varepsilon}^{-}+r\varphi_{2,\varepsilon}^{-} -(a+s)\varphi_{1,\varepsilon}^{-} 
      \end{align*}
      whenever, $\mu_1<\frac{\varepsilon}{\|K\|_{\infty}}$.
      Therefore, 
      \begin{equation}\label{WE-8}
          \lambda_p(\mu\circ\mathcal{K}+\mathcal{A})\ge \lambda_{\mu_2}^{a+s,e}-\varepsilon\quad \text{whenever}\ 0<\mu_1<\frac{\varepsilon}{\|K\|_{\infty}}.
      \end{equation}
      By \eqref{WE-7} and \eqref{WE-8}, we obtain that $\lambda_p(\mu\circ\mathcal{K}+\mathcal{A})\to \lambda_{\mu_2}^{a+s,e}$ as $\mu_1\to 0^+$, which yields the desired result.
      
       Next, we suppose that $ \lambda_{\mu_2}^{a+s,e}=-(a+s)_{\min}$. Since the function $L(\nu)$ defined by \eqref{L-equation} is strictly decreasing,  $\lambda_p\Big(\mu_2\mathcal{K}+\Big(\frac{rs}{\varepsilon-(a+s)_{\min}+a+s}-e\Big)\Big)<\varepsilon-(a+s)_{\min}$ for every $\varepsilon>0$. Therefore, as in \eqref{WE-7}, we have that 
      \begin{equation*}
           \lambda_p(\mu\circ\mathcal{K}+\mathcal{A})\le \varepsilon-(a+s)_{\min}+\varepsilon\Big\|1+\frac{s}{\varepsilon-(a+s)_{\min}+a+s}\Big\|_{\infty} \quad \text{whenever}\ 0<\mu_1\ll 1,
  \end{equation*}
  which yields that 
  \begin{equation} \label{WE-9}   \limsup_{\mu_1\to0^+}\lambda_p(\mu\circ\mathcal{K}+\mathcal{A})\le- (a+s)_{\min}.
  \end{equation}
  On the other hand, for every $\mu_1>0$ and $\varepsilon>0$, it follows from Theorem \ref{TH4}, that there is $\varphi^{\varepsilon}=(\varphi^{\varepsilon}_{1},\varphi_{2}^{\varepsilon})\in X^{++}\times X^{++}$, such that 
  $$ 
(\lambda_p(\mu\circ\mathcal{K}+\mathcal{A})+\varepsilon)\varphi^{\varepsilon}\ge \mu\circ\mathcal{K}\varphi^{\varepsilon}+\mathcal{A}\varphi^{\varepsilon},
  $$
  which implies 
  $$ 
(\lambda_p(\mu\circ\mathcal{K}+\mathcal{A})+\varepsilon)\varphi^{\varepsilon}_1\ge \mu_1\mathcal{K}\varphi^{\varepsilon}_1+r\varphi_2^{\varepsilon}-(a+s)\varphi_1^{\varepsilon}\ge \mu_1\mathcal{K}\varphi^{\varepsilon}_1-(a+s)\varphi_1^{\varepsilon},
  $$
  since $r\varphi_2^{\varepsilon}\ge 0$. Therefore, 
  $$ 
  \lambda_p(\mu\circ\mathcal{K}+\mathcal{A})\ge \lambda_p(\mu_1\mathcal{K}-(a+s)\mathcal{I})
  $$
   since $\varepsilon>0$ is arbitrarily chosen and $\varphi_{1}^{\varepsilon}\in X^+$.  However, by \cite[Theorem 2.2 (iv)]{SX2015},
   $$ \lim_{\mu_1\to 0^+}\lambda_p(\mu_1\mathcal{K}-(a+s)\mathcal{I})=-(a+s)_{\min}. $$
   As a result, 
   \begin{equation*}  \liminf_{\mu_1\to0^+}\lambda_p(\mu\circ\mathcal{K}+\mathcal{A})\ge\liminf_{\mu_1\to0^+}\lambda_p(\mu_1\mathcal{K}-(a+s)\mathcal{I}) =-(a+s)_{\min},
   \end{equation*}
   which combined with \eqref{WE-9}  yield 
   $$   \lim_{\mu_1\to0^+}\lambda_p(\mu\circ\mathcal{K}+\mathcal{A})= -(a+s)_{\min}.$$
   This completes the proof of the result.

   \medskip

   \noindent {\rm (ii)} It follows by proper modification of the proof of {\rm (i)}.
  \end{proof}

   \section{Proofs of Proposition \ref{Main-pop2} and Theorem \ref{Th2-4} }\label{Sec6}

We first give a proof of Proposition \ref{Main-pop2} and one more intermediate result (see Lemma \ref{Lem-6} below).

  \begin{proof}[Proof of Proposition \ref{Main-pop2}]  Fix $\xi>0$, $q, z\in X^+\setminus\{0\}$ and $l,h\in X^+$. Since the resolvent function is locally analytic, then the function $(\lambda_p(\xi\mathcal{K}-l\mathcal{I}),\infty)\ni \nu\mapsto \Psi_{\xi,l}(\cdot,\nu,z)$ is smooth. Moreover, it follows from the resolvent identity that

\begin{equation}\label{Pl3-1}
      \frac{\partial \Psi_{\xi,l}(\cdot,\nu,z)}{\partial{\nu}}=-(\nu\mathcal{I}-(\xi\mathcal{K}-l\mathcal{I}))^{-2}(z)=-\Psi_{\xi,l}(\cdot,\nu,\Psi_{\xi,l}(\cdot,\nu,z)).
  \end{equation} 
  Thus the function 
  \begin{equation}\label{IL3}
      \tilde{\Psi}_{\xi,l}^{h,q,z}(\nu)=\int_{\Omega}q(x)\Psi_{\xi,l}(x,\nu,z)-\int_{\Omega}h-\nu|\Omega|,\quad \nu>\lambda_p(\xi\mathcal{K}-l\mathcal{I})
  \end{equation}
 is smooth with derivative function given by    \begin{equation}\label{Pl3-2}
      \frac{d\tilde{\Psi}_{\xi,l}^{h,q,z}(\nu)}{d\nu}=-\int_{\Omega}q(x)\Psi_{\xi,l}(x,\nu,\Psi_{\xi,l}(\cdot,\nu,z))dx-|\Omega|\quad \nu>\lambda_p(\xi\mathcal{K}-l\mathcal{I}).
  \end{equation}
  But, since $\xi\mathcal{K}-l\mathcal{I} $ generates a  strongly continuous and strongly positive  semigroup $\{e^{t(\xi\mathcal{K}-l(\cdot))\mathcal{I} )}\}_{t\ge 0}$, and $p\in X^{+}\setminus\{0\}$, then $\Psi_{\xi,l}(\cdot,\nu,p) \in X^{++}$, hence $\Psi_{\xi,l}(\cdot,\nu,\Psi_{\xi,l}(\cdot,\nu,z))\in X^{++}$. Thus, in view of the fact that $q\in X^{+}\setminus\{0\}$, it follows from \eqref{Pl3-2} that $ \frac{d\tilde{\Psi}_{\xi,l}^{h,q,z}(\nu)}{d\nu}<0$ for all $\nu>\lambda_p(\xi\mathcal{K}-l\mathcal{I})$. Furthermore,    \begin{equation}\label{Pl4-1}
    \lim_{\nu\to\infty}\tilde{\Psi}_{\xi,l}^{h,q,z}(\nu)=-\infty.
\end{equation} 
Indeed, an integration gives $$ 
\nu\int_{\Omega}\Psi_{\xi,l}(x,\nu,z)dx=-\int_{\Omega}l(x)\Psi_{\xi,l}(x,\nu,z)dx+\int_{\Omega}z(x)dx\le \int_{\Omega}z(x)dx\quad \forall\ \nu>\lambda_p(\xi\mathcal{K}-l(\cdot)\mathcal{I}).
$$ 
  Observing that $\lambda_p(\xi\mathcal{K}-l\mathcal{I})\le \lambda_p(\xi\mathcal{K})= 0$, then 
\begin{align*}
\tilde{\Psi}_{\xi,l}^{h,q,z}(\nu)\le & \|q\|_{\infty}\int_{\Omega}{\Psi}_{\xi,l}(x,\nu,z)dx-\int_{\Omega}h(x)dx-\nu|\Omega|\cr 
\le &\frac{\|q\|_{\infty}}{\nu}\int_{\Omega}z(x)dx-\int_{\Omega}h(x)dx-\nu|\Omega|,
\end{align*} 
for every   $\nu>0$. Hence \eqref{Pl4-1} holds   since the right hand side of the last inequality converges to negative infinity as $\nu$ approaches infinity. As a result, since $\tilde{\Psi}_{\xi,l}^{h,q,z}(\nu)$ is the expression at the right side of the equation \eqref{IL1}, therefore, the algebraic equation \eqref{IL1} has a (unique) root if and only if \eqref{IL2} holds. \eqref{Pl1} easily holds by inspection. Next, we proceed to prove assertions {\rm (i)} and {\rm (ii)}.

  \quad {\rm (i)} Suppose that $\lambda_p(\xi\mathcal{K}-l\mathcal{I})$ is an eigenvalue of $\xi\mathcal{K}-l\mathcal{I}$ with a positive function. We claim that  \begin{equation}\label{Pl4-2}
    \lim_{\nu\to\lambda_p^+(\xi\mathcal{K}-l\mathcal{I})}\tilde{\Psi}_{\xi,l}^{h,q,z}(\nu)=\infty.
\end{equation}
Indeed, if both \eqref{Pl4-1} and \eqref{Pl4-2} hold, then \eqref{IL2} holds, and by the intermediate value theorem and the strict monotonicity of $\tilde{\Psi}_{\xi,l}^{h,q,z} $ in $\nu$, there is a unique  $\nu>\lambda_p(\xi\mathcal{K}-l\mathcal{I})$ satisfying $\tilde{\Psi}_{\xi,l}^{h,q,z}(\nu)=0$. Now, we proceed to show that \eqref{Pl4-2}  hold. From \eqref{Pl6-2}, for every $\nu>\lambda_p(\xi\mathcal{K}-l\mathcal{I})$, we have 
\begin{align*}
    \Psi_{\xi,l}(\cdot,\nu,z)=&\int_0^1e^{-\nu t}e^{t(\xi\mathcal{K}-l\mathcal{I})}zdt+\int_{1}^{\infty}e^{-\nu t}e^{t(\xi\mathcal{K}-l\mathcal{I})}zdt\cr 
    =& \int_0^1e^{-\nu t}e^{t(\xi\mathcal{K}-l\mathcal{I})}zdt+e^{-\nu}\int_{0}^{\infty}e^{-\nu t}e^{t(\xi\mathcal{K}-l\mathcal{I})}e^{\mu_2\mathcal{K}-e\mathcal{I}}zdt\cr 
     =& \int_0^1e^{-\nu t}e^{t(\xi\mathcal{K}-l\mathcal{I})}zdt+e^{-\nu}\Psi_{\xi,l}(\cdot,\nu,\tilde{z}).
\end{align*}
where $\tilde{z}:=e^{(\xi\mathcal{K}-l\mathcal{I})}z$. Thus, since $\{e^{t(\xi\mathcal{K}-l\mathcal{I})}\}_{t\ge 0} $ is strongly positive and $z\in X^{+}\setminus\{0\}$,  $\tilde{z}\in X^{++}$; that is  $\tilde{z}_{\min}>0$, and   
\begin{equation*}
    \Psi_{\xi,l}(\cdot,\nu,z)\ge e^{-\nu}\Psi_{\xi,l}(\cdot,\nu,\tilde{z})\ge \tilde{z}_{\min}e^{-\nu}\Psi_{\xi,l}(\cdot,\nu,1).
\end{equation*} 
Hence,
\begin{equation}\label{Pl6-3}
    \tilde{\Psi}_{\xi,l}^{h,q,z}(\nu)=\int_{\Omega}q(x)\Psi_{\xi,l}(x,\nu,z)dx-\int_{\Omega}h-\nu|\Omega|\ge \tilde{z}_{\min}e^{-\nu}\int_{\Omega}q(x)\Psi_{\xi,l}(x,\nu,1)dx -\int_{\Omega}h-\nu|\Omega|.
\end{equation}
Now, let $\psi\in X^{++}$ be an eigenfunction of $\lambda_p(\xi\mathcal{K}-l\mathcal{I})$ satisfying $\psi_{\max}=1$. Then 
$$ 
e^{t(\xi\mathcal{K}-l\mathcal{I})}\psi=e^{t\lambda_p(\xi\mathcal{K}-l\mathcal{I})}\psi\quad \forall\ t>0.
$$
Hence, since $0<\psi\le 1$ and $\{e^{t(\xi\mathcal{K}-l\mathcal{I})}\}_{t\ge }$ is positive, we obtain that 
$$
\Psi_{\xi,l}(\cdot,\nu,1)=\int_{0}^{\infty}e^{-\nu t}e^{t(\xi\mathcal{K}-l\mathcal{I})}(t)dt\ge \int_0^{\infty}e^{-(\nu-\lambda_p(\xi\mathcal{K}-l\mathcal{I}))}\psi=\frac{1}{\nu-\lambda_p(\xi\mathcal{K}-e\mathcal{I})}\psi(\cdot)\quad \forall\ \nu>\lambda_p(\xi\mathcal{K}-l\mathcal{I}).
$$
This implies that 
$$
\int_{\Omega}q(x)\Psi_{\xi,l}(x,\nu,1)dx\ge \frac{1}{\nu-\lambda_p(\xi\mathcal{K}-l\mathcal{I})}\int_{\Omega}q(x)\psi(x)dx\to \infty\quad \text{as}\ \nu\to\lambda_p^+(\xi\mathcal{K}-l\mathcal{I}).
$$
(Note that we have used the fact that $\int_{\Omega}q(x)\psi(x)dx>0$ since $\psi\in X^{++}$ and $q\in X^{+}\setminus\{0\}$.) This together with  inequality \eqref{Pl6-3} implies that  \eqref{Pl4-2} holds. Clearly, by \eqref{Pl4-2}, we have that \eqref{IL2} holds.

\quad {\rm (ii)} Suppose that $\lambda_p(\xi\mathcal{K}-l\mathcal{I})<0$ and \eqref{IL2} holds. Then, $0\in (\lambda_p(\xi\mathcal{K}-l\mathcal{I}),\infty)$ and $\tilde{\Psi}_{\xi,l}^{h,q,z}(\tilde{\lambda}_{q,\xi,z}^{h,l})=0$. Therefore, since $\tilde{\Psi}_{\xi,l}^{h,q,z}$ is strictly decreasing, then  $\tilde{\lambda}_{q,\xi,z}^{h,l} $ and $\tilde{\Psi}_{\xi,l}^{h,q,z}(0)=\sigma_{q,\xi,z}^{h,l} $ have the same sign.

  \end{proof}

\noindent   For future use, we set
 \begin{equation}
     (\Psi_{\mu_2,e}(\cdot,\nu,s\mathcal{I})r)w(x)=\Psi_{\mu_2,e}(x,\nu,sw)r(x)\quad \forall\ w\in X, \ x\in\overline{\Omega}, \quad \text{and}\ \nu>\lambda_p(\mu_2\mathcal{K}-e\mathcal{I}).
 \end{equation}
 The following result will also be needed to complete the proof of Theorem \ref{Th2-4}.

  \begin{lemma}\label{Lem-6}  Let $\nu>\lambda_p(\mu_2\mathcal{K}-e\mathcal{I})$ be fixed, so that the bounded linear map $\nu\mathcal{I}-(\mu_2\mathcal{K}-e\mathcal{I})$ is invertible and let  $\Psi_{\mu_2,e}(\cdot,\nu,\cdot)$ be given by \eqref{Pl6-2}. For every $\mu_1>0$, let $ 
\lambda_p(\mu_1\mathcal{K}+\Psi_{\mu_2,e}(\cdot,\nu,s\mathcal{I})r-(a+s)\mathcal{I})$ 
      denote the principal spectrum point of the bounded linear operator $\mu_1\mathcal{K}+\Psi_{\mu_2,e}(\cdot,\nu;s\mathcal{I})r-(a+s)\mathcal{I}$ on $X$. Then there is $\mu_1^*\gg 0$, 
 such that for every $\mu_1>\mu_1^{*}$, $\lambda_p(\mu_1\mathcal{K}+\Psi_{\mu_2,e}(\cdot,\nu;s\mathcal{I})r-(a+s)\mathcal{I})$ is a geometrically simple eigenvalue whose eigenspace is spanned by a strictly positive eigenfunction. Furthermore, 
      \begin{equation}  \label{Pl9-1}
      \lim_{\mu_1\to\infty}\lambda_p(\mu_1\mathcal{K}+\Psi_{\mu_2,e}(\cdot,\nu,s\mathcal{I})r-(a+s)\mathcal{I})
          =\frac{1}{|\Omega|}\left(\int_{\Omega}\Psi_{\mu_2,e}(\cdot,\nu,s)rdx-\int_{\Omega}(a+s)dx\right).
      \end{equation}
      In particular,  if  \eqref{IL2} holds for the choices of $\xi=\mu_2$, $q=r$, $ z=s$, $l=e$ and $h=a+s$, then  $\nu:=\tilde{\lambda}_{r,\mu_2,s}^{a+s,e}$, the unique zero of \eqref{IL1}, satisfies 
      \begin{equation}\label{Pl9-9}
\lim_{\mu_1\to\infty}\lambda_p(\mu_1\mathcal{K}+\Psi_{\mu_2,e}(\cdot,\tilde{\lambda}_{r,\mu_2,s}^{a+s,e},s\mathcal{I})r-(a+s)\mathcal{I})
          =\tilde{\lambda}_{r,\mu_2,s}^{a+s,e}.
      \end{equation}
      
  \end{lemma}
  \begin{proof} Let $\nu>\lambda_p(\mu_2\mathcal{K}-e(\cdot)\mathcal{I})$ be fixed. Taking $w\equiv 1$, for every $\mu_1>0$, we have 
  \begin{align*}
      \mu_1\mathcal{K}(w)(x)+ \Psi_{\mu_2,e} (x,\nu,sw)r(x)-(a(x)+s(x))w(x)& = \Psi_{\mu_2,e}(x,\nu,sw)r(x)-(a(x)+s(x))w(x)\cr 
      &\ge  -\|a+s\|_{\infty}w(x) \quad \forall\ x\in\overline{\Omega},
  \end{align*}
  since $r\ge 0$, $s\ge 0$ and $\Psi_{\mu_2,e}(\cdot,\nu,\cdot)$ is positive. Thus 
  \begin{equation}\label{Pl8-1}
   \lambda_p(\mu_1\mathcal{K}+\Psi_{\mu_2,e}(\cdot,\nu,s\mathcal{I})r-(a+s)\mathcal{I})   \ge  -\|a+s\|_{\infty}\quad\forall\ \mu_1>0.
  \end{equation}
  Recall that 
  \begin{align}\label{Pl8-2}
     &((\mu_1\mathcal{K}+\Psi_{\mu_2,e}(\cdot,\nu,s\mathcal{I})r-(a+s)\mathcal{I})w)(x)\cr =&\mu_1\int_{\Omega}\kappa(x,y)w(y)dy -\mu_1K(x)w(x)+ \Psi_{\mu_2,e}(x,\nu;sw)r(x)-(a(x)+s(x))w(x)\quad \forall\ w\in X, \ x\in\overline{\Omega}.\cr
  \end{align}
  Observe also that, with $w\equiv 1$, 
  \begin{align*}
     &-\mu_1K(x)w(x)+\Psi_{\mu_2,e}(x,\nu;sw)r(x)-(a(x)+s(x))w(x)\cr
     \le &-\mu_1K_{\min}+ \Psi_{\mu_2,e}(x,\nu,s)r(x)-(a(x)+s(x))\cr 
     \le& (-\mu_1K_{\min}+\|\Psi_{\mu_2,e}(\cdot,\nu,s)r\|_{\infty}-(a+s)_{\min})w(x)\quad \forall\ x\in\overline{\Omega}, \ \mu_1>0.
  \end{align*}
  Hence, for every $\mu_1>0$, it holds that
  $$ 
  \lambda_p(-\mu_1K\mathcal{I}+\Psi_{\mu_2,e}(\cdot,\nu;s\mathcal{I})r-(a+s)\mathcal{I})\le -\mu_1K_{\min}+\|\Psi_{\mu_2,e}(\cdot,\nu,s)r\|_{\infty}-(a+s)_{\min}.
  $$
  Therefore, in view of \eqref{Pl8-1}, we have that 
$$ 
   \lambda_p(\mu_1\mathcal{K}+\Psi_{\mu_2,e}(\cdot,\nu,s\mathcal{I})r-(a+s)\mathcal{I})   >\lambda_p(-\mu_1K\mathcal{I}+\Psi_{\mu_2,e}(\cdot,\nu;s\mathcal{I})r-(a+s)\mathcal{I})
$$
  whenever 
  $$ \mu_1> \mu_1^*:=\frac{\|\Psi_{\mu_2,e}(\cdot,\nu,s)\|_{\infty}+\|a+s\|_{\infty}-(a+s)_{\min}}{K_{\min}}. $$ 
  Hence, since the mapping  
  $$X\ni w\mapsto \mu_1\int_{\Omega}\kappa(\cdot,y)w(y)dy\in X, $$
  is strictly positive and compact, and  $\mu_1\mathcal{K}+\Psi_{\mu_2,e}(\cdot,\nu,s\mathcal{I})r-(a+s)\mathcal{I}$ is a compact perturbation of $\Psi_{\mu_2,e}(\cdot,\nu,s\mathcal{I})r-(a+s)\mathcal{I}$, it follows by similar arguments as in the proof \cite[Proposition 3.8]{SX2015} that $\lambda_p(\mu_1\mathcal{K}+\Psi_{\mu_2,e}(\cdot,\nu,s\mathcal{I})r-(a+s)\mathcal{I})$ is the principal eigenvalue of $\mu_1\mathcal{K}+\Psi_{\mu_2,e}(\cdot,\nu,s\mathcal{I})r-(a+s)\mathcal{I}$ whose eigenspace is spanned by a strictly positive eigenfunction in $X^{++}$ for every $\mu_1>\mu_1^*$. Now, it remains to show that \eqref{Pl9-1} holds.

  For every $\mu_1>\mu_1^*$, let $\psi^{\mu_1}$ be the positive  principal eigenfunction of  $\lambda_p(\mu_1\mathcal{K}+\Psi_{\mu_2,e}(\cdot,\nu;s\mathcal{I})r-(a+s)\mathcal{I})$  satisfying $\|\psi^{\mu_1}\|_{L^2(\Omega)} = 1$. It follows by computations similar to that leading to inequality \eqref{HK24-2-1}  that 
  \begin{equation*}      \lim_{\mu_1\to\infty}\Big\|\psi^{\mu_1}-\frac{1}{|\Omega|}\int_{\Omega}\psi^{\mu_1}\Big\|_{L^2(\Omega)}=0.
  \end{equation*}
  which together with the fact that $\|\psi^{\mu_1}\|_{L^2(\Omega)}=1$ for every $\mu_1>\mu_1^*$ yield 
  \begin{equation}\label{Pl9-4}
      \lim_{\mu_1\to\infty}\frac{1}{|\Omega|}\int_{\Omega}\psi^{\mu_1}dx=\frac{1}{\sqrt{|\Omega|}} \quad \text{and}\quad \lim_{\mu_1\to\infty}\Big\|\psi^{\mu_1}-\frac{1}{\sqrt{|\Omega|}}\Big\|_{L^2(\Omega)}=0.
  \end{equation}
  Now, integrating the equation satisfied by $\psi^{\mu_1}$ on $\Omega$ yield
  \begin{equation}\label{Pl9-6}
      \lambda_p(\mu_1\mathcal{K}+\Psi_{\mu_2,e}(\cdot,\nu;s\mathcal{I})r-(a+s)\mathcal{I})=\frac{\int_{\Omega}\Psi_{\mu_2,e}(x,\nu,s\psi^{\mu_1})r(x)dx-\int_{\Omega}(a(x)+s(x))\psi^{\mu_1}dx}{\int_{\Omega}\psi^{\mu_1}(x)dx}.
  \end{equation}
  for every $\mu_1>\mu_1^*$. Note that, since $\kappa$ is symmetric, it follows from \cite[Proposition 3.9]{SX2015} that 
  \begin{equation}\label{Pl9-2}
      \lambda_p(\mu_2\mathcal{K}-e\mathcal{I})=\sup_{w\in L^2(\Omega)\setminus\{0\}}\frac{\int_{\Omega}\int_{\Omega}\mu_2\kappa(x,y)[w(y)-w(x)]w(x)dydx-\int_{\Omega}e(x)w^2(x)dx}{\int_{\Omega}w^2dx}.
  \end{equation}
  Observe that, for every $\mu_1>\mu_1^*,$ the function  
  $$W_{\mu_1}(\cdot):=\Psi_{\mu_2,e}(\cdot,\nu,s\psi^{\mu_1})-\frac{1}{\sqrt{|\Omega|}}\Psi_{\mu_2,e}(\cdot,\nu,s)=\Psi_{\mu_2,e}\Big(\cdot,\nu,\Big(\psi^{\mu_1}-\frac{1}{\sqrt{|\Omega|}}\Big)s\Big)$$
  satisfies,
  \begin{equation*}
      \nu W_{\mu_1}=\mu_2\mathcal{K}W_{\mu_1}-eW_{\mu_1}+\Big(\psi^{\mu_1}-\frac{1}{\sqrt{|\Omega|}}\Big)s
  \end{equation*}
  Multiplying this equation by $W_{\mu_1}$ and integrating the resulting equation on $\Omega$ and recalling \eqref{Pl9-2}, we obtain 
  $$ 
  \nu\int_{\Omega}W^2_{\mu_1}(x)dx\leq \lambda_p(\mu_2\mathcal{K}-e\mathcal{I})\int_{\Omega}W_{\mu_1}^2(x)dx+\int_{\Omega}\Big(\psi^{\mu_1}-\frac{1}{\sqrt{|\Omega|}}\Big)s(x)W_{\mu_1}(x)dx,
  $$
equivalently,
$$ 
\Big(\nu-\lambda_p(\mu_2\mathcal{K}-e\mathcal{I})\Big)\|W_{\mu_1}\|_{L^2(\Omega)}^{2}\le \int_{\Omega}\Big(\psi^{\mu_1}-\frac{1}{\sqrt{|\Omega|}}\Big)s(x)W_{\mu_1}(x)dx.
$$
  Therefore, since $\nu>\lambda_p(\mu_2\mathcal{K}-e\mathcal{I})$, applying H\"older's inequality to the right hand side of the last inequality, we obtain 
  
  \begin{align*}
\left\| \Psi_{\mu_2,e}(\cdot,\nu,s\psi^{\mu_1})-\frac{1}{\sqrt{|\Omega|}}\Psi_{\mu_2,e}(\cdot,\nu,s)\right\|_{L^2(\Omega)}\le& \frac{\Big\|\Big(\psi^{\mu_1}-\frac{1}{\sqrt{|\Omega|}}\Big)s\Big\|_{L^2(\Omega)}
 }{\nu-\lambda_p(\mu_2\mathcal{K}-e\mathcal{I})} \cr
 \leq & \frac{\|s\|_{\infty}\Big\|\psi^{\mu_1}-\frac{1}{\sqrt{|\Omega|}}\Big\|_{L^2(\Omega)}
 }{\nu-\lambda_p(\mu_2\mathcal{K}-e\mathcal{I})} .
 \end{align*}
 As a consequence, we deduce from \eqref{Pl9-4} that 
 $$ 
 \left\| \Psi_{\mu_2,e}(\cdot,\nu,s\psi^{\mu_1})-\frac{1}{\sqrt{|\Omega|}}\Psi_{\mu_2,e}(\cdot,\nu,s)\right\|_{L^2(\Omega)}\to 0 \ \text{as}\ \mu_1\to\infty
. $$
Hence
$$ 
\lim_{\mu_1\to\infty}\int_{\Omega}\Psi_{\mu_2,e}(x,\nu,s\psi^{\mu_1})r(x)dx=\frac{1}{\sqrt{|\Omega|}}\int_{\Omega}\Psi_{\mu_2,e}(x,\nu,s)r(x)dx.
$$
Note also from \eqref{Pl9-4} that 
$$ 
\lim_{\mu_1\to\infty}\int_{\Omega}(a(x)+s(x))\psi^{\mu_1}(x)dx=\frac{1}{\sqrt{|\Omega|}}\int_{\Omega}(a(x)+s(x))dx \quad \text{and}\quad \lim_{\mu_1\to\infty}\int_{\Omega}\psi^{\mu_1}dx=\frac{
|\Omega|}{\sqrt{|\Omega|}}.
$$
Therefore, letting $\mu_1\to\infty$ in \eqref{Pl9-6} we obtain \eqref{Pl9-1}. Finally,  if \eqref{IL2} holds for $\xi=\mu_2$, $q=r$, $z=s$, $l=e$, and $h=a+s$, and  $\nu=\tilde{\lambda}^{a+s,e}_{r,\mu_2,s}$, then \eqref{Pl9-9} holds  by \eqref{Pl1} and \eqref{Pl9-1}. This completes the proof of the lemma.
  \end{proof}

\noindent  Thanks to Lemma \ref{Lem-6} and Proposition \ref{Main-pop2}, we can now present the proof of Theorem \ref{Th2-4}.

  \begin{proof}[Proof of Theorem \ref{Th2-4}] We only present the proof of {\rm (i)} since {\rm (ii)} can be proved by similar arguments.  Let $\xi=\mu_2$, $q=r$, $z=s$, $l=e$, and $h=a+s$. We distinguish two cases. 

  \medskip

\noindent{\bf Case 1.} In this case, we suppose that \eqref{IL2} holds. Let $\tilde{\lambda}^{a+s,e}_{r,\mu_2,s}>\lambda_{p}(\mu_2\mathcal{K}-e\mathcal{I})$ be the unique zero of \eqref{IL1}.   By Lemma \ref{Lem-6}, there is $\mu_1^*\gg 1$ such that $ \lambda_p(\mu_1\mathcal{K}+\Psi_{\mu_2,e}(\cdot,\tilde{\lambda}^{a+s,e}_{r,\mu_2,s},s\mathcal{I})r-(a+s)\mathcal{I})$ is the principal eigenvalue of $\mu_1\mathcal{K}+\Psi_{\mu_2,e}(\cdot,\tilde{\lambda}^{a+s,e}_{r,\mu_2,s},s\mathcal{I})r-(a+s)\mathcal{I}$ with a strictly positive eigenfunction $\varphi_1^{\mu_1}$. Set $\varphi_2^{\mu_2}=\Psi_{\mu_2,e}(\cdot,\tilde{\lambda}^{a+s,e}_{r,\mu_2,s},s\varphi_1^{\mu_1})$. Then, for every $\mu_1>\mu_1^*$, $(\varphi_1^{\mu_1},\varphi_{2}^{\mu_2})\in X^{++}\times X^{++}$ and satisfies 
  \begin{equation*}
      \begin{cases}
          \lambda_p(\mu_1\mathcal{K}+\Psi_{\mu_2,e}(\cdot,\tilde{\lambda}^{a+s,e}_{r,\mu_2,s},s\mathcal{I})r-(a+s)\mathcal{I})\varphi_1^{\mu_1}=\mu_1\mathcal{K}\varphi_{1}^{\mu_1}+r\varphi_{2}^{\mu_2}-(a+s)\varphi_{1}^{\mu_1}\cr 
          \tilde{\lambda}^{a+s,e}_{r,\mu_2,s}\varphi_2^{\mu_2}=\mu_2\mathcal{K}\varphi_2^{\mu_2}+s\varphi_{1}^{\mu_1}-e\varphi_{2}^{\mu_2}.
      \end{cases}
  \end{equation*}
  Hence, it follows from Theorem \ref{TH4} that 
  \begin{align*}
 & \min\big\{\lambda_p(\mu_1\mathcal{K}+\Psi_{\mu_2,e}(\cdot,\tilde{\lambda}^{a+s,e}_{r,\mu_2,s},s\mathcal{I})r-(a+s)\mathcal{I}),\tilde{\lambda}^{a+s,e}_{r,\mu_2,s}\big\}\cr 
 \le & \lambda_p(\mu\circ\mathcal{K}+\mathcal{A})
 \le  \max\big\{\lambda_p(\mu_1\mathcal{K}+\Psi_{\mu_2,e}(\cdot,\tilde{\lambda}^{a+s,e}_{r,\mu_2,s},s\mathcal{I})r-(a+s)\mathcal{I}),\tilde{\lambda}^{a+s,e}_{r,\mu_2,s}\big\}
  \end{align*}
     for every $\mu_1>\mu_1^*$.  Therefore, letting $\mu_1\to\infty$ in the last inequalities and recalling \eqref{Pl9-9}, we obtain that 
$$
\lim_{\mu_1\to\infty}\lambda_p(\mu\circ\mathcal{K}+\mathcal{A})=\tilde{\lambda}^{a+s,e}_{r,\mu_2,s}, 
$$
which yields the desired result.

\medskip

\noindent{\bf Case 2}. Here we suppose that \eqref{IL2}. Hence, $\tilde{\Psi}_{\mu_2,e}^{a+e,r,s}(\nu)<0$ for every $\nu>\lambda_p(\mu_2\mathcal{K}-e\mathcal{I})$,  where $\tilde{\Psi}_{\mu_2,e}^{a+e,r,s}$ is defined by \eqref{IL3}. This together with \eqref{Pl9-1} implies  that for every $\nu>\lambda_{p}(\mu_2\mathcal{K}-e\mathcal{I})$, 
\begin{equation}  \lim_{\mu_1\to\infty}\lambda_p(\mu_1\mathcal{K}+\Psi_{\mu_2,e}(\cdot,\nu,s\mathcal{I})r-(a+s)\mathcal{I})
          =\frac{1}{|\Omega|}\left(\int_{\Omega}\Psi_{\mu_2,e}(\cdot,\nu,s)rdx-\int_{\Omega}(a+s)dx\right)<  \nu
\end{equation}
Therefore, for every $\nu>\lambda_{p}(\mu_2\mathcal{K}-e\mathcal{I})$, there is $\mu_1^{1,\nu}>0$ such that

\begin{equation}\label{IL4}
 \lambda_p(\mu_1\mathcal{K}+\Psi_{\mu_2,e}(\cdot,\nu,s\mathcal{I})r-(a+s)\mathcal{I})<\nu \quad \mu_1>\mu_{1}^{1,\nu},\ \nu>\lambda_p(\mu_2\mathcal{K}-e\mathcal{I}).   
\end{equation}
Furthermore, by Lemma \ref{Lem-6}, for every $\nu>\lambda_p(\mu_2\mathcal{K}-e\mathcal{I})$, there is $\mu_1^{2,\nu}\ge \mu_1^{1,\nu}$, such that $\lambda_p(\mu_1\mathcal{K}+\Psi_{\mu_2,e}(\cdot,\nu;s\mathcal{I})r-(a+s)\mathcal{I})$ is a geometrically simple eigenvalue whose eigenspace is spanned by a strictly positive eigenfunction. Now, fix $\nu>\lambda_p(\mu_2\mathcal{K}-e\mathcal{I})$. For every $\mu_1>\mu_1^{2,\nu}$, with $\mu=(\mu_1,\mu_2)$ let $\varphi_1^{\mu}$ be the strictly positive eigenfunction of $ \lambda_p(\mu_1\mathcal{K}+\Psi_{\mu_2,e}(\cdot,\nu;s\mathcal{I})r-(a+s)\mathcal{I})$ satisfying $\max_{x\in\overline{\Omega}}\varphi_1^{\mu}(x)=1$ and set $\varphi_2^{\mu}=\Psi_{\mu_2,e}(\cdot,\nu,s\varphi_1^{\mu})$. Then $(\varphi_1^{\mu},\varphi_2^{\mu})\in X^{++}\times X^{++}$ and  by \eqref{IL4}, 
$$
\begin{cases}
    \nu\varphi_{1}^{\mu}>\mu_1\mathcal{K}\varphi_1+r\varphi_2^{\mu}-(a+s)\varphi_1^{\mu_1}\cr 
    \nu\varphi_{2}^{\mu}=\mu_2\mathcal{K}\varphi_2^{\mu}+s\varphi_1^{\mu}-e\varphi_2^{\mu}.
\end{cases}
$$
It then follows from Theorem \ref{TH4} that 
$$
\lambda_p(\mu\circ\mathcal{K}+\mathcal{A})\le \nu\quad \forall\ \mu_1>\mu_1^{2,\nu}.
$$
Therefore, since $\nu>\lambda_p(\mu_2\mathcal{K}-e\mathcal{I})$ is arbitrary, 
\begin{equation}\label{IL6}
    \limsup_{\mu_1\to\infty}\lambda_p(\mu\circ\mathcal{K}+\mathcal{A})\le \lambda_p(\mu_2\mathcal{K}-e\mathcal{I}).
\end{equation}
On the other hand, for every $\varepsilon>0$ and dispersal rate $\mu=(\mu_1,\mu_2)$, with $\varphi=(\varphi_1,\varphi_2)\in X^{++}\times X^{++}$ given by \eqref{IL9}, it follows from \eqref{IL10} that 
$$ 
(\lambda_p(\mu\circ\mathcal{K}-\mathcal{A})+\varepsilon)\varphi_2>\mu_2\mathcal{K}\varphi_2+s\varphi_1-e\varphi_2\ge \mu_2\mathcal{K}\varphi_2-e\varphi_2. 
$$
It then follows from \eqref{sup-inf-char} that 
 $
\lambda_p(\mu\circ\mathcal{K}-\mathcal{A})+\varepsilon\ge \lambda_p(\mu_2\mathcal{K}-e\mathcal{I}).
 $ Since $\varepsilon$ is arbitrary, then for any dispersal rate $\mu=(\mu_1,\mu_2)$, it always holds that 
$ 
\lambda_p(\mu\circ\mathcal{K}-\mathcal{A})\ge \lambda_p(\mu_2\mathcal{K}-e\mathcal{I}). 
 $ Therefore,
\begin{equation}\label{IL6-2}
    \liminf_{\mu_1\to\infty}\lambda_p(\mu\circ\mathcal{K}+\mathcal{A})\ge \lambda_p(\mu_2\mathcal{K}-e\mathcal{I}).
\end{equation}
We deduce the desired result from \eqref{IL6} and \eqref{IL6-2}.

  \end{proof}

  \section{Proofs  of Theorems  \ref{Th2-5} and \ref{Th2-6}}\label{Sec7}

\begin{proof}[Proof of Theorem \ref{Th2-5}]{\rm (i)} Suppose that $r_{\min}>0$. Let $\eta_1^*$ be defined as in \eqref{Th2-5-Eq1}. We distinguish two cases. 

\noindent{\bf Case 1.} $\eta_1^*>-(a+s)_{\min}$. Consider the function 
$$\mathcal{L}\ :\ (-(a+s)_{\min},\infty)\ni \eta\mapsto\int_{\Omega}\left(\frac{rs}{\eta+a+s}-(e+\eta)\right). $$ 

\noindent Clearly, $\mathcal{L}$ is strictly decreasing in $\eta$ and $\eta_1^*=\inf\{\eta\in(-(a+s)_{\min},\infty)\ :\ \mathcal{L}(\eta)<0\}$. Note that $\mathcal{L}(\eta)\to-\infty$ as $\eta\to\infty$. Observe that,  since $\eta_1^*>-(a+s)_{\min}$,  $\mathcal{L}(\eta^*_1)=0$, which implies that 
$$ 
\eta_1^*=\frac{1}{|\Omega|}\int_{\Omega}\frac{rs}{\eta_1^*+a+s}dx-\hat{e}>-\hat{e}.
$$ 
Thus, $\eta_1^*>-\min\{(a+s)_{\min},\hat{e}\}$. Fix $\eta>-\min\{(a+s)_{\min},\hat{e}\}$. Since 
$$ 
\lim_{\mu_1\to 0}\lambda_p(\mu_1\mathcal{K}-(a+s)\mathcal{I})=\max_{x\in\overline{\Omega}}\{-(a(x)+s(x))\}=-(a+s)_{\min},
$$ then  there exists $ \mu_1^{\eta}>0$ such that $\eta>\lambda_p(\mu_1\mathcal{K}-(a+s)\mathcal{I})$ for every $0<\mu_1<\mu_1^{\eta}$.  Hence $\eta\mathcal{I}-(\mu_1\mathcal{K}-(a+s)\mathcal{I})$ is invertible for every $0<\mu_1<\mu_1^{\eta}$.  For every $0<\mu_1<\mu_1^{\eta}$, let 
\begin{equation}\label{4-Eq2}
\varphi_1:=\left(\eta-(\mu_1\mathcal{K}-(a+s))\right)^{-1}(r).
\end{equation}
Hence, $\varphi_1\in X^{++}$ (since $r\in X^{+}\setminus\{0\}$) and satisfies 
\begin{equation}\label{4-Eq1}
    \eta\varphi_1=\mu_1\mathcal{K}(\varphi_1)+r -(a+s)\varphi_1
\end{equation}
Note from the comparison principle for nonlocal operators that $\|\varphi_1\|_{\infty}\le \frac{\|r\|_{\infty}}{\eta+(a+s)_{\min}}$. Moreover, it follows from \eqref{4-Eq1} that 
\begin{equation*}
    \varphi_1-\frac{r}{\eta+a+s}=\frac{\mu_1\mathcal{K}(\varphi_1)}{\eta+a+s}
\end{equation*}
from which we deduce that 
\begin{equation}\label{4-Eq10}
    \left\|\varphi_1-\frac{r}{\eta+a+s} \right\|_{\infty}\le \mu_1\left\|\frac{\mathcal{K}(\varphi_1)}{\eta+a+s}\right\|_{\infty}\le \mu_1\frac{\|\mathcal{K}\|\|\varphi_1\|_{\infty}}{\eta+(a+s)_{\min}}\le \mu_1\frac{\|\mathcal{K}\|\|r\|_{\infty}}{(\eta+(a+s)_{\min})^2}\to 0 \ \text{as}\ \mu_1\to0.
\end{equation}
Next, since
$$
\lim_{\mu_2\to\infty}\lambda_p(\mu_2\mathcal{K}-e\mathcal{I})=-\hat{e},
$$
then there is $\mu_2^{\eta}>0$ such that $\eta>\lambda_p(\mu_2\mathcal{K}-e\mathcal{I})$ for every $\mu_2>\mu_2^{\eta}$. For every $0<\mu_1<\mu_1^{\eta}$ and $\mu_2>\mu_2^{\eta}$, let $$
\varphi_2:=(\eta\mathcal{I}-(\mu_2\mathcal{K}-e\mathcal{I}))^{-1}(s\varphi_1)
$$
where $\varphi_1$ is given by \eqref{4-Eq2}. Since $s\varphi_1\in X^+\setminus\{0\}$, then $\varphi_2\in X^{++}$. Now, observe that $(\varphi_1,\varphi_2)$ satisfies
 
\begin{equation}\label{4-Eq3}
\begin{cases}
\eta\varphi_1=\mu_1\mathcal{K}(\varphi_1)+r\varphi_2-(a+s)\varphi_1 & x\in\Omega,\cr 
    \eta\varphi_2=\mu_2\mathcal{K}(\varphi_2)+s\varphi_1-e\varphi_2 & x\in \Omega.
    \end{cases}
\end{equation}
Observe that from the equation 
\begin{equation}\label{4-Eq4}
    \eta\varphi_2=\mu_2\mathcal{K}(\varphi_2)+s\varphi_1-e\varphi_2
\end{equation}
that 
\begin{equation}\label{4-Eq5}
\varphi_2=\hat{\varphi_2}+\frac{1}{K(x)}\left[\int_{\Omega}\kappa(x,y)(\varphi_2(y)-\hat{\varphi_2})dy-\frac{1}{\mu_2}\left((\eta+e)\varphi_2-s\varphi_1\right)\right].
\end{equation}
Thus, multiplying \eqref{4-Eq4} by $\varphi_2$, integrating the resulting equation and employing arguments similar to that leading to \eqref{HK24-2-1}, we get that 
\begin{align*}
    \left\|\varphi_2-\hat{\varphi_2}\right\|_{L^2(\Omega)}^2\le& \frac{1}{\mu_2\beta_*}\left(\int_{\Omega}s\varphi_1\varphi_2-\eta\int_{\Omega}\varphi_2^2\right)\le\frac{(\|s\|_{\infty}+|\eta|)}{\mu_2\beta_*}\sum_{i=1}^2\int_{\Omega}\varphi_i^2.
\end{align*}
Therefore, it follows from \eqref{4-Eq5} that 
\begin{align}\label{4-Eq6}
    \|\varphi_2-\hat{\varphi_2}\|_{\infty}\le& \frac{1}{K_{\min}}\left(\kappa_{\max}\int_{\Omega}|\varphi_2-\hat{\varphi_2}|dy+\frac{(\|\eta+e\|_{\infty}+\|s\|_{\infty})}{\mu_2}\sum_{i=1}^2\|\varphi_i\|_{\infty} \right)\cr 
    \le & \frac{1}{K_{\min}}\left(\kappa_{\max}\sqrt{|\Omega|}\|\varphi_2-\hat{\varphi_2}\|_{L^2(\Omega)}+\frac{(\|\eta+e\|_{\infty}+\|s\|_{\infty})}{\mu_2}\sum_{i=1}^2\|\varphi_i\|_{\infty} \right)\cr 
    \le &\frac{1}{K_{\min}}\left(\frac{\kappa_{\max}\sqrt{|\Omega|(\|s\|_{\infty}+|\eta|)}}{\sqrt{\mu_2\beta_*}}\sqrt{\sum_{i=1}^2|\Omega|\|\varphi_i\|_{\infty}^2} +\frac{(\|\eta+e\|_{\infty}+\|s\|_{\infty})}{\mu_2}\sum_{i=1}^2\|\varphi_i\|_{\infty}\right)\cr 
    \le &\frac{1}{K_{\min}}\left(\frac{\kappa_{\max}|\Omega|\sqrt{2(\|s\|_{\infty}+|\eta|)}}{\sqrt{\mu_2\beta_*}}+\frac{(\|\eta+e\|_{\infty}+\|s\|_{\infty})}{\mu_2} \right)\sum_{i=1}^2\|\varphi_i\|_{\infty}.
\end{align}
Next, we claim that
\begin{equation}\label{4-Eq7}
    \limsup_{\mu_2\to\infty,\mu_1\to 0}\|\varphi_2\|_{\infty}<\infty.
\end{equation}
We proceed by contradiction to establish that \eqref{4-Eq7} holds. Suppose that there is a sequence $\mu_{2}^n\to \infty$ and $\mu_{1}^n\to 0$ such that \begin{equation}
\lim_{n\to\infty}\|\varphi_{2,n}\|_{\infty}=\infty,
\end{equation}
where $(\varphi_{1,n},\varphi_{2,n})$ solves \eqref{4-Eq3} with $\mu=(\mu_{1,n},\mu_{2,n})$ for every $n\ge1$. Set $\psi_{i,n}=\frac{\varphi_{i, n}}{\|\varphi_{2,n}\|_{\infty}}, i = 1,2$ for every $n\ge 1$. Hence, 
\begin{equation}\label{4-Eq8}
    \|\psi_{1,n}\|_{\infty}\le \frac{\|\varphi_{1,n}\|_{\infty}}{\|\varphi_{2,n}\|_{\infty}}\leq \frac{\|r\|_{\infty}}{(\eta+(a+s)_{\min})\|\varphi_{2,n}\|_{\infty}}\to 0 \ \text{as}\ n\to\infty
\end{equation}
and 
\begin{equation}\label{4-Eq9}
    \|\psi_{2,n}\|_{\infty}=1\quad \forall\ n\ge 1.
\end{equation}
Thus, dividing both sides of \eqref{4-Eq6} by $\|\varphi_{2,n}\|_{\infty}$, we obtain
$$ 
\left\|\psi_{2,n}-\frac{1}{|\Omega|}\int_{\Omega}{\psi_{2,n}}\right\|_{\infty}\to 0 \quad \text{as}\ n\to\infty,
$$
which together with \eqref{4-Eq9} yields 
\begin{equation}
    \lim_{n\to\infty}\frac{1}{|\Omega|}\int_{\Omega}\psi_{2,n}=1 \quad \text{and}\quad \lim_{n\to\infty}\|\psi_{2,n}-1\|_{\infty}=0.
\end{equation}
If we integrate both sides of \eqref{4-Eq4}, we obtain that
$$ 
\eta\frac{1}{|\Omega|}\int_{\Omega}\psi_{2,n}(x)dx=-\frac{1}{|\Omega|}\int_{\Omega}e(x)\psi_{2,n}(x)dx+\frac{1}{|\Omega|}\int_{\Omega}s(x)\psi_{1,n}(x)dx.
$$
Letting $n\to\infty$ in this equation and recalling \eqref{4-Eq8} and \eqref{4-Eq9}, we obtain that $\eta=-\hat{e}$, which contradicts our initial assumption that $\eta>-\hat{e}$. Therefore, \eqref{4-Eq7} holds. 
\quad Now, thanks to \eqref{4-Eq7} and the fact that $\|\varphi_1\|_{\infty}\leq \frac{\|r\|_{\infty}}{\eta+(a+s)_{\min}}$, we conclude from \eqref{4-Eq6} that 
\begin{equation*}
    \lim_{\mu_1\to0,\mu_2\to\infty} \|\varphi_2-\hat{\varphi_2}\|_{\infty}=0.
\end{equation*}
As a result, if we integrate \eqref{4-Eq4} and recall \eqref{4-Eq10}, we get 
\begin{equation}\label{4-Eq11}
    \varphi_2\to \frac{\int_{\Omega}\frac{sr}{\eta+a+s}}{|\Omega|(\eta+\hat{e})}=1+\frac{\mathcal{L}(\eta)}{|\Omega|(\eta+\hat{e})} \quad \text{as}\ \mu_1\to0 \ \text{and}\ \mu_2\to\infty\ \text{uniformly in }\ \Omega.
\end{equation}
From this point, we distinguish two subcases.

\noindent{\bf Subcase 1.} $\eta>\eta_1^*$. Then $\mathcal{L}(\eta)<\mathcal{L}(\eta_1^*)=0$. Then, it follows from \eqref{4-Eq11} that there is $M\gg1$ such that $\varphi_2<1$ for every $0<\mu_1<\frac{1}{M}$ and $\mu_2>M$. As a result, we conclude from \eqref{4-Eq3} that 
\begin{equation*}
\begin{cases}
\eta\varphi_1\ge\mu_1\mathcal{K}(\varphi_1)+r\varphi_2-(a+s)\varphi_1 & x\in\Omega,\cr 
    \eta\varphi_2\ge \mu_2\mathcal{K}(\varphi_2)+s\varphi_1-e\varphi_2 & x\in \Omega,
    \end{cases}
\end{equation*}
for every $0<\mu_1<\frac{1}{M}$ and $\mu_2>M$, which implies that 
$$ \limsup_{\mu_1\to0,\mu_2\to\infty}\lambda_p(\mu\circ\mathcal{K}+\mathcal{A})\le \eta.$$
Since $\eta>\eta_1^*$ is arbitrary, then $ \limsup_{\mu_1\to0,\mu_2\to\infty}\lambda_p(\mu\circ\mathcal{K}+\mathcal{A})\le \eta_1^*.$

\noindent{\bf Subcase 2.} $\eta<\eta_1^*$. Then $\mathcal{L}(\eta)>\mathcal{L}(\eta_1^*)=0$. Then, it follows from \eqref{4-Eq11} that there is $M\gg1$ such that $\varphi_2>1$ for every $0<\mu_1<\frac{1}{M}$ and $\mu_2>M$. As a result, we conclude from \eqref{4-Eq3} that 
\begin{equation*}
\begin{cases}
\eta\varphi_1\le\mu_1\mathcal{K}(\varphi_1)+r\varphi_2-(a+s)\varphi_1 & x\in\Omega,\cr 
    \eta\varphi_2\le \mu_2\mathcal{K}(\varphi_2)+s\varphi_1-e\varphi_2 & x\in \Omega,
    \end{cases}
\end{equation*}
for every $0<\mu_1<\frac{1}{M}$ and $\mu_2>M$, which implies that 
$$ \liminf_{\mu_1\to0,\mu_2\to\infty}\lambda_p(\mu\circ\mathcal{K}+\mathcal{A})\ge \eta.$$
Since $\eta<\eta_1^*$ is arbitrary, then $ \liminf_{\mu_1\to0,\mu_2\to\infty}\lambda_p(\mu\circ\mathcal{K}+\mathcal{A})\ge \eta_1^*.$

\quad From Subcases 1 and 2, we deduce that $ \lim_{\mu_1\to0,\mu_2\to\infty}\lambda_p(\mu\circ\mathcal{K}+\mathcal{A})= \eta_1^*.$

\noindent{\bf Case 2.} $\eta_1^*=-(a+s)_{\min}$. In this case, we must have that $-(a+s)_{\min}\ge -\hat{e}$. Indeed, if this was false, that is $-\hat{e}>-(a+s)_{\min}$, we would have $$\mathcal{L}(-\hat{e})=\int_{\Omega}\left(\frac{rs}{a+s-\hat{e}}-(e-\hat{e})\right)=\int_{\Omega}\frac{rs}{a+s-\hat{e}}>0,
$$ 
from which we deduce that $\eta_1^*>-\hat{e}$. So, we obtain a contradiction to the fact that $\eta_1^*=-(a+s)_{\min}<-\hat{e}$. Next, since $\eta_1^*=-(a+s)_{\min}=-\min\{(a+s)_{\min},\hat{e}\}$, then $\mathcal{L}(\eta)<0$ for every $\eta>-(a+s)_{\min}$. It then follows from \eqref{4-Eq11} and subcase 1 above that $$ \limsup_{\mu_1\to0,\mu_2\to\infty}\lambda_p(\mu\circ\mathcal{K}+\mathcal{A})\le \eta_1^*.$$ Now, it remains to show that 
\begin{equation}\label{4-Eq12}
 \liminf_{\mu_1\to0,\mu_2\to\infty}\lambda_p(\mu\circ\mathcal{K}+\mathcal{A})\ge \eta_1^*.
\end{equation}
Let $\varepsilon>0$. By Theorem \ref{TH4}, for every $\mu_1>0$ and $\mu_2>0$, there is $\varphi\in X^{++}\times X^{++}$ such that 
$$
(\lambda_p(\mu\circ\mathcal{K}+\mathcal{A})+\varepsilon)\varphi_1\ge \mu_1\mathcal{K}(\varphi_1)+r\varphi_2-(a+s)\varphi_1\ge \mu_1\mathcal{K}(\varphi_1)-(a+s)\varphi_1
$$
and
$$
(\lambda_p(\mu\circ\mathcal{K}+\mathcal{A})+\varepsilon)\varphi_2\ge \mu_2\mathcal{K}(\varphi_2)+s\varphi_1-e\varphi_2\ge \mu_2\mathcal{K}(\varphi_2)-e\varphi_2.
$$
Hence,
$$
\lambda_p(\mu\circ\mathcal{K}+\mathcal{A})+\varepsilon\ge \lambda_p(\mu_1\mathcal{K}-(a+s)\mathcal{I}) \quad \forall\ \mu_1>0,\ \mu_2>0.
$$
Therefore, since $\varepsilon$ is arbitrary,
\begin{align*}
\liminf_{\mu_1\to0,\mu_2\to\infty}\lambda_p(\mu\circ\mathcal{K}+\mathcal{A})\ge&\lim_{\mu_1\to0}\lambda_p(\mu_1\mathcal{K}-(a+s)\mathcal{I})
=-(a+s)_{\min}=\eta_1^*,
\end{align*}
which yields \eqref{4-Eq12}.  Finally, it follows from the monotonicity of the function $\mathcal{L}$ that $\eta^*_1$ and $\lim_{\eta\to 0^+}\mathcal{L}(\eta)$ have the same sign.

\medskip

\quad {\rm (ii)} It follows by a proper modification of the proof of {\rm (i)}.

\end{proof}

We complete this section with a proof of Theorem \ref{Th2-6}.

\begin{proof}[Proof of Theorem \ref{Th2-6}] Suppose that $r_{\min}>0$.\\ 
{\rm (i)}  Suppose on the contrary that there exist a sequence $\mu^n:=(\mu_1^{n},\mu_2^{n})$ of dispersal rates with $\mu_1^n\to 0^+$ as $n\to\infty$ such that 
\begin{equation}\label{4-Eq13}  \limsup_{n\to\infty}\lambda_p(\mu^n\circ\mathcal{K}+\mathcal{A})\le 0.
\end{equation}
We now distinguish three cases.

{\bf Case 1.} In this case, we suppose without loss of generality that $\mu_2^n\to\infty$ as $n\to\infty$. Therefore, it follows from Theorem \ref{Th2-5}-{\rm (i)} that $\lambda_p(\mu^n\circ\mathcal{K}+\mathcal{A})\to\eta_1^*$ as $n\to\infty$. This contradicts with \eqref{4-Eq13} since $\eta_1^*>0$. 

{\bf Case 2.} In this case, we suppose without loss of generality that $\mu_2^n\to 0$ as $n\to\infty$. Therefore, it follows from Theorem \ref{Th2-1}-{\rm (i)} that $\lambda_p(\mu^n\circ\mathcal{K}+\mathcal{A})\to\Lambda_{\max}$ as $n\to\infty$. Hence $\Lambda_{\max}\le 0$ by \eqref{4-Eq13}. As a result, we obtain from \eqref{HK8} that 
$$
\sqrt{(a(x)+s(x)-e(x))^2+4r(x)s(x)}\le (a(x)+s(x)+e(x)) \quad \forall\ x\in\Omega,
$$
which is equivalent to 
$$
r(x)s(x)\le (a(x)+s(x))e(x) \quad \forall\ x\in\Omega.
$$
Thus for every $\eta>-(a+s)_{\min}$,
$$
\int_{\Omega}\left(\frac{rs}{\eta+a+s}-(\eta+e)\right)\le \int_{\Omega}\left(\frac{(a+s)e}{\eta+a+s}-(e+\eta)\right)=-\eta\int_{\Omega}\left(\frac{ e}{\eta+a+s}+1\right).
$$
In particular, since $\eta_1^*>0\ge-(a+s)_{\min}$, then 
$$
0=\int_{\Omega}\left(\frac{rs}{\eta_1^*+a+s}-(\eta_1^*+e)\right)\le -\eta_1^*\int_{\Omega}\left(\frac{ e}{\eta_1^*+a+s}+1\right)
$$
This is clearly impossible since $\eta_1^*>0$. 

{\bf Case 3.} In this case, we suppose without loss of generality that $\mu_2^n\to\mu_2$ as $n\to \infty$ for some positive number $\mu_2$. Therefore, it follows from Theorem \ref{Th2-3}-{\rm (i)} that $\lambda_p(\mu^n\circ\mathcal{K}+\mathcal{A})\to\lambda_{\mu_2}^{a+s,e}$ as $n\to\infty$, which in view of \eqref{4-Eq13} implies that 
\begin{equation}\label{4-Eq14}
 \lambda_{\mu_2}^{a+s,e}\le 0.   
\end{equation}However,  
$$
0=\frac{1}{|\Omega|}\int_{\Omega}\left(\frac{rs}{\eta_1^*+a+s}-(\eta_1^*+e)\right)\le \lambda_p\Big(\mu_2\mathcal{K}+\Big(\frac{rs}{\eta_1^*+a+s}-(e+\eta_1^*)\Big)\mathcal{I}\Big).
$$
Hence, it follows from Proposition \ref{Main-prop1} and the fact that the mapping $(-(a+s)_{\min},\infty)\ni \eta\mapsto \lambda_p\Big(\mu_2\mathcal{K}+\Big(\frac{rs}{\eta+a+s}-(e+\eta)\Big)\mathcal{I}\Big)$ is strictly decreasing that $\lambda_{\mu_2}^{a+s,e}\ge \eta_1^*$. This clearly contradicts with \eqref{4-Eq14} since $\eta_1^*>0$.

From cases 1, 2 and 3, we deduce that there is $\delta_1^*>0$ such the statement of Theorem \ref{Th2-6}-{\rm (i)} holds.

{\rm (ii)} Suppose that $\Lambda_{\max}<0$. Then $\eta_1^*<0$.  Indeed, if it was the case that $\eta_{1}^*\ge 0$, then 
$$
\int_{\Omega}\left(\frac{rs}{\eta_1^*+a+s}-(\eta_1^*+e)\right)=0,
$$
which would imply  that there is some $x_0\in\overline{\Omega}$ such that
\begin{equation*}
    r(x_0)s(x_0)\ge (\eta_1^*+e(x_0))(\eta_1^*+(a(x_0)+s(x_0))) \ge e(x_0)(a(x_0)+s(x_0)).
\end{equation*}
This implies that $\Lambda_{\max}\ge 0$, which is not true. Hence, we must have that $\eta_1^*<0$. Now, we proceed by contradiction to establish the desired result.  Suppose to the contrary that there exist a sequence $\mu^n:=(\mu_1^{n},\mu_2^{n})$ of dispersal rates with $\mu_1^n\to 0^+$ as $n\to\infty$ such that 
\begin{equation}\label{4-Eq15}  \liminf_{n\to\infty}\lambda_p(\mu^n\circ\mathcal{K}+\mathcal{A})\ge 0.
\end{equation}
We distinguish three cases.

{\bf Case 1.} In this case, we suppose without loss of generality that $\mu_2^n\to\infty$ as $n\to\infty$. Therefore, it follows from Theorem \ref{Th2-5}-{\rm (i)} that $\lambda_p(\mu^n\circ\mathcal{K}+\mathcal{A})\to\eta_1^*$ as $n\to\infty$. This contradicts with \eqref{4-Eq15} since $\eta_1^*<0$.

{\bf Case 2.} In this case, we suppose without loss of generality that $\mu_2^n\to0$ as $n\to\infty $. Therefore, it follows from Theorem \ref{Th2-1}-{\rm (i)} that $\lambda_p(\mu^n\circ\mathcal{K}+\mathcal{A})\to\Lambda_{\max}$ as $n\to\infty$. This contradicts with \eqref{4-Eq15} since $\Lambda_{\max}<0$.

{\bf Case 3.} In this case, we suppose without loss of generality that $\mu_2^n\to\mu_2$ as $n\to\infty $ for some positive number $\mu_2$. Therefore, it follows from Theorem \ref{Th2-3}-{\rm (i)} that $\lambda_p(\mu^n\circ\mathcal{K}+\mathcal{A})\to\lambda_{\mu_2}^{a+s,e}$ as $n\to\infty$, which in view of \eqref{4-Eq15} implies that 
\begin{equation}\label{4-Eq16}
 \lambda_{\mu_2}^{a+s,e}\ge 0.   
\end{equation}
However,  since $ \eta_1^*<0$, then $-(a+s)_{\min}<0\le \lambda_{\mu_2}^{a+s,e}$. Therefore, by Proposition \ref{Main-prop1}, 
$$
0=\lambda_p\Big(\mu_2\mathcal{K}+\Big(\frac{rs}{\lambda_{\mu_2}^{a+s,e}+a+s}-(e+\lambda_{\mu_2}^{a+s,e})\Big)\mathcal{I}\Big),
$$
which implies that there is some $x_1\in\overline{\Omega}$ such that 
\begin{equation*}
    r(x_1)s(x_1)\ge (\lambda_{\mu_2}^{a+s,e}+e(x_1))(\lambda_{\mu_2}^{a+s,e}+(a(x_1)+s(x_1))) > e(x_1)(a(x_1)+s(x_1)).
\end{equation*}
This implies that $\Lambda_{\max}> 0$, which is not true.

From cases 1, 2 and 3, we deduce that there is $\delta_1^*>0$ such the statement of Theorem \ref{Th2-6}-{\rm (ii)} holds.

{\rm (iii)} Suppose that $(a+s)_{\min}>0$ and $\eta_1^*<0<\Lambda_{\max}$. Then 
$$ 
\int_{\Omega}\left(\frac{rs}{a+s}-e\right)<0<\left(\frac{rs}{a+s}-e\right)_{\max}.
$$
This implies that the function $\frac{rs}{a+s}-e$ is not constant. Hence, the function $(0,\infty)\ni\mu_2\mapsto \lambda_{p}\Big(\mu_2^*\mathcal{K}+\Big(\frac{rs}{a+s}-e\Big)\mathcal{I}\Big)$ is strictly decreasing  (\cite[Theorem 2.2(i)]{SX2015}).  Moreover, since
$$ 
\lim_{\mu_2\to\infty}\lambda_{p}\Big(\mu_2^*\mathcal{K}+\Big(\frac{rs}{a+s}-e\Big)\mathcal{I}\Big)=\frac{1}{|\Omega|}\int_{\Omega}\Big(\frac{rs}{a+s}-e\Big)$$  \text{and} $$\lim_{\mu_2\to 0}\lambda_{p}\Big(\mu_2^*\mathcal{K}+\Big(\frac{rs}{a+s}-e\Big)\mathcal{I}\Big)=\Big(\frac{rs}{a+s}-e\Big)_{\max},
$$
 there is a unique $\mu_2^*>0$ such that $\lambda_{p}\Big(\mu_2^*\mathcal{K}+\Big(\frac{rs}{a+s}-e\Big)\mathcal{I}\Big)=0$. Next, fix $0<\varepsilon<\mu_2^*$ and we proceed in two steps to complete the proof of the theorem.

 {\bf Step 1.} In this step, we show that there exist $\delta_{1,\varepsilon}>0$ and $m^*>0$ such that $\lambda_p(\mu\circ\mathcal{K}+\mathcal{A})>m^*$ for every $\mu=(\mu_1,\mu_2)\in (0,\delta_{1,\varepsilon})\times(0,\mu_2^*-\varepsilon)$. If this was false, there would exist a sequence $\mu^n=(\mu_1^n,\mu_2^n)$ of dispersal rates with $0<\mu_2^n\le\mu_2^*-\varepsilon$ for every $n\ge 1$ and $\mu_1^n\to0^+$  as $n\to\infty$ such that \eqref{4-Eq13} holds. If, up to a subsequence, $\mu_2^n\to 0$ as $n\to \infty$, then by Theorem \ref{Th2-1}, $\lambda_p(\mu^n\circ\mathcal{K}+\mathcal{A})\to \Lambda_{\max}>0$  as $n\to\infty$. This contradicts with \eqref{4-Eq13}. On the other hand, if $\mu_2^n\to \mu_2\in(0,\mu_2^*-\varepsilon]$ (up to a subsequence) as $n\to\infty$, then by Theorem \ref{Th2-3}-{\rm (i)}, $\lambda_p(\mu^n\circ\mathcal{K}+\mathcal{A})\to\lambda_{\mu_2}^{a+s,e}$ as $n\to\infty$. Hence, in view of \eqref{4-Eq13}, we must have that $\lambda_{\mu_2}^{a+s,e}\le 0$. However, since $\mu_2<\mu_2^*$, then $\lambda_p\Big(\mu_2\mathcal{K}+\Big(\frac{rs}{a+s}-e\Big)\mathcal{I}\Big)>\lambda_p\Big(\mu_2^*\mathcal{K}+\Big(\frac{rs}{a+s}-e\Big)\mathcal{I}\Big)=0$, which implies that $\lambda_{\mu_2}^{a+s,e}>0$, so we get a contradiction. Therefore, we conclude that there exist $\delta_{1,\varepsilon}>0$ and $m^*>0$ such that $\lambda_p(\mu\circ\mathcal{K}+\mathcal{A})>m^*$ for every $\mu=(\mu_1,\mu_2)\in (0,\delta_{1,\varepsilon})\times(0,\mu_2^*-\varepsilon)$.

 {\bf Step 2.} In this step, we show that there exist $\delta_{1,\varepsilon}>0$ and $m^*>0$ such that $\lambda_p(\mu\circ\mathcal{K}+\mathcal{A})<-m^*$ for every $\mu=(\mu_1,\mu_2)\in (0,\delta_{1,\varepsilon})\times(\mu_2^*+\varepsilon,\infty)$. If this was false, there would exist a sequence $\mu^n=(\mu_1^n,\mu_2^n)$ of dispersal rates with $\mu_2^n\ge\mu_2^*-\varepsilon$ for every $n\ge 1$ and $\mu_1^n\to0^+$  as $n\to\infty$ such that \eqref{4-Eq15} holds.

 If, up to a subsequence, $\mu_2^n\to \infty$ as $n\to \infty$, then by Theorem \ref{Th2-5}-{\rm (i)}, $\lambda_p(\mu^n\circ\mathcal{K}+\mathcal{A})\to \eta_1^*<0$  as $n\to\infty$. This contradicts with \eqref{4-Eq15}. On the other hand, if $\mu_2^n\to \mu_2\in[\mu_2^*+\varepsilon,\infty)$ (up to a subsequence) as $n\to\infty$, then by Theorem \ref{Th2-3}-{\rm (i)}, $\lambda_p(\mu^n\circ\mathcal{K}+\mathcal{A})\to\lambda_{\mu_2}^{a+s,e}$ as $n\to\infty$.  Hence, in view of \eqref{4-Eq15}, we must have that $\lambda_{\mu_2}^{a+s,e}\ge 0$. However, since $\mu_2>\mu_2^*$, then $\lambda_p\Big(\mu_2\mathcal{K}+\Big(\frac{rs}{a+s}-e\Big)\mathcal{I}\Big)<\lambda_p\Big(\mu_2^*\mathcal{K}+\Big(\frac{rs}{a+s}-e\Big)\mathcal{I}\Big)=0$, which implies that $\lambda_{\mu_2}^{a+s,e}<0$, so we get a contradiction. Therefore, we conclude that there exist $\delta_{1,\varepsilon}>0$ and $m^*>0$ such that $\lambda_p(\mu\circ\mathcal{K}+\mathcal{A})<-m^*$ for every $\mu=(\mu_1,\mu_2)\in (0,\delta_{1,\varepsilon})\times(\mu_2^*+\varepsilon,\infty)$. This completes the proof of the theorem.

\end{proof}

\section{Proofs of Theorems \ref{TH6}, \ref{TH7}, \ref{TH8}, and \ref{TH9}}\label{Sec8}

\begin{proof}[Proof of Theorems \ref{TH6}] We proceed by contradiction. To this end, suppose that there exists a sequence of dispersal rate $\mu^n=(\mu_1^n,\mu_2^n)$ converging to zero, a sequence of positive steady states solutions ${\bf u}^n$ of \eqref{model} associated with the dispersal rates $\mu^n$, and  a sequence $\{x_n\}_{n\ge 1}$ of elements of $\overline{\Omega}$ such that 
\begin{equation}\label{BH1}
    \inf_{n\ge 1}\|{\bf u}^n(x^n)-{\bf V}(x^n)\|>0.
\end{equation}

\noindent{\bf Step 1.} Uniform bound on $
\{{\bf u}^n\}_{n\ge1 }$.  It is clear that any positive steady state solution ${\bf u}$ of \eqref{model} is a subsolution of the cooperative  and subhomogeneous system

\begin{equation}\label{BH2}
    \begin{cases}
        0=\mu_1\mathcal{K}u_1+r(x)u_2-(a(x)+s(x)+b(x)u_1(x))u_1(x) & x\in\Omega,\cr 
        0=\mu_2\mathcal{K}u_2+s(x)u_1-(e(x)+f(x)u_2)u_2 & x\in\Omega.
    \end{cases}
\end{equation}
Now, since $\Lambda_{\max}>0$, it follows from Theorem \ref{Th2-1} that there is $\eta_0>0$, such that $\lambda_p(\mu\circ\mathcal{K}+\mathcal{A})>0$ for every $\mu=(\mu_1,\mu_2)$ with $0<\mu_i<\eta_0$, $i=1,2$. Hence, by \cite[Theorem 4-(i)]{OSUU2023}, system \eqref{BH2} has a unique globally stable positive steady state solution ${\bf u}^{\mu}$ for every $\mu=(\mu_1,\mu_2)$ with $0<\mu_i<\eta_0$, $i=1,2$.  However, an easy computation shows that the constant vector function ${\bf M}^*:=(M_1^*,M_2^*)$ where $M_1^*=M_2^*=\max\{\frac{r_{\max}}{b_{\min}},\frac{s_{\max}}{f_{\min}}\}$ is a supersolution of \eqref{BH2}. Hence, by stability of ${\bf u}^{\mu}$, we deduce that ${\bf u}^{\mu}\le_1 M^*$ for every $\mu=(\mu_1,\mu_2)$ with $0<\mu_i<\eta_0$, $i=1,2$. Since ${\bf u}^n$ is a subsolution of \eqref{BH2} with $\mu^n=\mu$ and $\mu^n\to 0$ as $n\to \infty$, then, without loss of generality, we have that ${\bf u}^n\le_1 {\bf M}^*$ for every $n\ge 1$. 

\noindent{\bf Step 2.} For every $x\in\overline{\Omega}$ and $n\ge 1$, it holds that 
\begin{equation}\label{BH3}
    \Lambda(x)\le \max\{u_1^n(x),u_2^n(x)\}(b(x)+f(x)+c(x)+g(x)+\mu_1^nK(x)+\mu_2^nK(x)).
\end{equation}
Indeed, observe that ${\bf u}^n$ satisfies
$$ 
\begin{cases}
(b(x)u_1^n(x)+c(x)u_2^n(x)+\mu_1K(x))u_1^n(x)=r(x)u_2^n(x)-(a(x)+s(x))u_1^n(x)+\mu_1\int_{\Omega}\kappa(x,y)u_1^n(y)dy\cr 
(f(x)u_2^n(x)+g(x)u_1^n(x)+\mu_2K(x))u_2^n(x)=s(x)u_1^n(x)-e(x)u_1^n(x)+\mu_2\int_{\Omega}\kappa(x,y)u_2^n(y)dy.
\end{cases}
$$
Hence, 
$$ 
\max\{u_1^n(x),u_2^n(x)\}(b(x)+f(x)+c(x)+g(x)+\mu_1^nK(x)+\mu_2^nK(x)){\bf u}^n(x)\ge \mathcal{A}(x){\bf u}^n(x).
$$
As a result, we deduce from Theorem \ref{TH4} that \eqref{BH3} holds since ${\bf u}^n(x)>0$.

\noindent{\bf Step 3.} In this step, we complete the proof of the theorem by deriving contradiction to \eqref{BH1}. By Step 1 and the Bolzano-Weierstrass theorem, if possible after passing to a subsequence, we may suppose that ${\bf u}^n(x^n)\to {\bf u}^{\infty}$ as $n\to\infty$ for some ${\bf u}^{\infty}\in \mathbb{R}^+\times\mathbb{R}^+$. Furthermore, since $\overline{\Omega}$ is compact, again after passing to a further subsequence, we may suppose that there is some $x_{\infty}\in\overline{\Omega}$ such that $x_n\to x_{\infty}$ as $n\to\infty$. Therefore, observing from Step 1 that $\mu_i^n\|\mathcal{K}u_i^n\|_{\infty}\to 0$ as $n\to\infty$, then ${\bf u}^{\infty}$ satisfies 

\begin{equation}\label{BH4}
    \begin{cases}
        0=r(x_{\infty})u_2^{\infty}-(a(x_{\infty})+s(x_{\infty})+b(x_{\infty})u_1^{\infty}+c(x_{\infty})u_2^{\infty})u_1^{\infty}\cr 
        0=s(x_{\infty})u_1^{\infty}-(e(x_{\infty})u_2^{\infty}+g(x_{\infty})u_1^{\infty})u_2^{\infty}.
    \end{cases}
\end{equation}
Now, we distinguish two cases. 

\noindent{\bf Case 1.} $\Lambda(x_\infty)\le 0$. In this case, it follows from \eqref{BH4} and \cite[Theorem 2]{OSUU2023} that ${\bf u}^{\infty}={\bf 0}$. Since ${\bf V}(x_{\infty})$ is also a nonnegative solution of \eqref{BH4}, we also have by \cite[Theorem 2]{OSUU2023} that ${\bf V}(x_{\infty})={\bf 0}$. This clearly contradicts with \eqref{BH1} since ${\bf V}(x_n)\to {\bf V}(x_{\infty})$ and ${\bf u}^n(x_n)\to {\bf u}^{\infty}$ as $n\to\infty$.

\noindent{\bf Case 2.} $\Lambda(x_\infty)>0$. From Step 2, we have that 
$$ 
0<\Lambda(x_{\infty})\le \max\{u_1^{\infty},u_2^{\infty}\}(b(x_{\infty})+f(x_{\infty})+c(x_{\infty})+g(x_{\infty})),
$$
which implies that ${\bf u}^{\infty}\ne {\bf 0}$. However, by \cite[Theorem 2]{OSUU2023}, ${\bf V}(x_{\infty})$ is the unique nonnegative solution of \eqref{BH4} different from ${\bf 0}$. Therefore, ${\bf V}(x_{\infty})={\bf u}^{\infty}$, which contradicts with \eqref{BH1}. 

\quad Thanks to cases 1 and 2,  the desired result holds.

\end{proof}

\begin{proof}[Proof of Theorem \ref{TH7}] Suppose that $\tilde{\Lambda}>0$ and let ${\bf u}^{\mu}$ be a positive steady-state solution of \eqref{model} for ${\mu}>>1$.  This gives 
 \begin{equation}\label{stationary}
     \begin{cases}
      0=\mu_1\int_{\Omega}\kappa(x,y)(u_1^{\mu}(y)-u_1^{\mu}(x))dy +ru_2^{\mu}-su_1^{\mu} -(a+bu_1^{\mu}+\tau cu_2^{\mu})u_1^{\mu}& x\in\Omega,\cr
      0=\mu_2\int_{\Omega}\kappa(x,y)(u_2^{\mu}(y)-u_2^{\mu}(x))dy +su_1^{\mu}-(e+fu_2^{\mu}+\tau gu_1^{\mu})u_2^{\mu}& x\in\Omega.
     \end{cases}
 \end{equation}

Set $v_i^{\mu}=u_i^{\mu}-\frac{1}{|\Omega|}\int_{\Omega}u_i^{\mu}$, $i=1,2$.

\noindent{\bf Step 1.} In this step, we show that for each $i=1,2$, 
\begin{equation}\label{BH6}
    \lim_{\mu\to \infty}\|v_i^{\mu}\|_{\infty}=0.
\end{equation}

Multiplying the first equation of  \eqref{stationary}  by $u_1^{\mu}$ and integrating the resulting equation, it follows by similar arguments leading to \eqref{HK24-2-1} that 
\begin{equation*}
    \|v_1^{\mu}\|_{L^2(\Omega)}^2\leq \frac{1}{\mu_1\beta_*}\int_{\Omega}ru_1^{\mu}u_2^{\mu}\le \frac{\|u_1^{\mu}\|_{\infty}\|u_{2}^{\mu}\|_{\infty}}{\mu_1\beta_*}\int_{\Omega}r.
\end{equation*}
Similarly, 
\begin{equation*}
    \|v_2^{\mu}\|_{L^2(\Omega)}^2\leq \frac{1}{\mu_2\beta_*}\int_{\Omega}su_1^{\mu}u_2^{\mu}\le \frac{\|u_1^{\mu}\|_{\infty}\|u_{2}^{\mu}\|_{\infty}}{\mu_2\beta_*}\int_{\Omega}s.
\end{equation*}
Recalling from Step 1 of the proof of Theorem \ref{TH6} that $\|u_i^{\mu}\|_{\infty}\le \max\{\frac{r_{\max}}{b_{\min}},\frac{s_{\max}}{f_{\min}}\}$, we deduce from the last two inequalities that 
\begin{equation}\label{BH5}
    \lim_{\mu\to \infty}\|v_i^{\mu}\|_{L^2(\Omega)}=0.
\end{equation}
Next, using \eqref{BH5}, we show that \eqref{BH6} holds. To this end, observe that 
$$ 
v_1^{\mu}=\frac{1}{K(x)}\int_{\Omega}\kappa(x,y)v_1^{\mu}(y)dy+\frac{1}{\mu_1K(x)}\left(ru_2^{\mu}-(a+s+bu_1^{\mu}+cu_2^{\mu})u_1^{\mu}\right).
$$
Hence, by the H\"older's inequality,
\begin{align*}
    \|v_1^{\mu}\|_{\infty}\leq\frac{\|\kappa\|_{\infty}|\Omega|^{\frac{1}{2}}\|v_1^{\mu}\|_{L^2(\Omega)}}{K_{\min}}+\frac{\left(\|r\|_{\infty}\|u_2^{\mu}\|_{\infty}+(\|a+s\|_{\infty}+\|b\|_{\infty}\|u_1^{\mu}\|_{\infty}+\|c\|_{\infty}\|u_2^{\mu}\|)\|u_1^{\mu}\|_{\infty}\right)}{\mu_1K_{\min}} 
\end{align*}
Similarly, 
\begin{align*}
    \|v_2^{\mu}\|_{\infty}\leq\frac{\|\kappa\|_{\infty}|\Omega|^{\frac{1}{2}}\|v_2^{\mu}\|_{L^2(\Omega)}}{K_{\min}}+\frac{\left(\|s\|_{\infty}\|u_1^{\mu}\|_{\infty}+(\|e\|_{\infty}+\|f\|_{\infty}\|u_2^{\mu}\|_{\infty}+\|g\|_{\infty}\|u_1^{\mu}\|)\|u_2^{\mu}\|_{\infty}\right)}{\mu_1K_{\min}}. 
\end{align*}
Therefore, since $\|u_i^{\mu}\|_{\infty}\le \max\{\frac{r_{\max}}{b_{\min}},\frac{s_{\max}}{f_{\min}}\}$, \eqref{BH6} follows from the last two inequalities.

\noindent{\bf Step 2.} We show that 
\begin{equation}\label{BH7}
    \lambda_{p}(\mu\circ\mathcal{K}+\mathcal{A})\le \max\{\|u_1^{\mu}\|_{\infty},\|u_2^{\mu}\|_{\infty}\}(\|b\|_{\infty}+\|e\|_{\infty}+\|c\|_{\infty}+\|g\|_{\infty}).
\end{equation}
Indeed, observe that 
$$ 
\begin{cases}
 \max\{\|u_1^{\mu}\|_{\infty},\|u_2^{\mu}\|_{\infty}\}(\|b\|_{\infty}+\|e\|_{\infty}+\|c\|_{\infty}+\|g\|_{\infty})u_1^{\mu}\ge (bu_1^{\mu}+cu_2^{\mu})u_1=\mu_1\mathcal{K}u_1^{\mu}+ru_2^{\mu}-(a+s)u_1^{\mu}\cr 
 \max\{\|u_1^{\mu}\|_{\infty},\|u_2^{\mu}\|_{\infty}\}(\|b\|_{\infty}+\|e\|_{\infty}+\|c\|_{\infty}+\|g\|_{\infty})u_{2}^{\mu}\ge (fu_2^{\mu}+gu_1^{\mu})u_2^{\mu}=\mu_2\mathcal{K}u_2^{\mu}+su_1^{\mu}-eu_2^{\mu}.
\end{cases}
$$
Therefore, \eqref{BH7} follows from Theorem \ref{TH4} since ${\bf u}^{\mu}\in X^{++}\times X^{++}$.

\noindent{\bf Step 3.} In this step, we complete the proof of the theorem. By integrating both equations in \eqref{model}, and using the fact that $\kappa$ is symmetric, we obtain
\begin{equation}\label{BH8}
\begin{cases}
    0=\int_{\Omega}ru_2^{\mu}-\int_{\Omega}(a+s+bu_1^{\mu}+cu_2^{\mu})u_1^{\mu}\cr 
    0=\int_{\Omega}su_1^{\mu}-\int_{\Omega}(e+fu_2^{\mu}+gu_1^{\mu})u_2^{\mu}
    \end{cases}
\end{equation}
  Now, since $\|u^{\mu}\|_{\infty}$ is uniformly bounded in $\mu$, then after passing to a subsequence, we may suppose, there is ${\bf Q}=(Q_1,Q_2)\in \mathbb{R}^+\times\mathbb{R}^+$ such that $(\frac{1}{|\Omega|}\int_{\Omega}u_1^{\mu},\frac{1}{|\Omega|}\int_{\Omega}u_2^{\mu})\to {\bf Q}$ as $\mu\to \infty$. It then follows from Step 1 that ${\bf u}^{\mu}\to {\bf Q}$ as $\mu\to\infty$ uniformly on $\overline{\Omega}$.   Recalling from Theorem \ref{Th2-4} that $\lambda_p(\mu\circ\mathcal{K}+\mathcal{A})\to$ $\tilde{\Lambda}$ as $\mu\to \infty$ and $\tilde{\Lambda} >0$, we deduce from Step 2 that ${\bf Q}\ne {\bf 0}$. Now, letting $\mu\to\infty$ in \eqref{BH8}, we obtain that 
  \begin{equation}
      \begin{cases}
          0=\hat{r}Q_2-(\hat{a}+\hat{s}+\hat{b}Q_1+\hat{c}Q_2)Q_1 \cr 
          0=\hat{s}Q_1-(\hat{e}+\hat{f}Q_2+\hat{g}Q_1)Q_2.
      \end{cases}
  \end{equation}
  This shows that ${\bf Q}$ is a positive solution of \eqref{TH7-eq1}. Since by \cite[Theorem 2]{OSUU2023}, $\hat{V}$ is the unique positive solution of \eqref{TH7-eq1}, then ${\bf Q}=\hat{\bf V}$, and the result follows.

\end{proof}

\begin{proof}[Proof of Theorem \ref{TH8}] Suppose that $(a+s)_{\min}>0$ and $r_{\min}>0$. Let $\mu_2>0$ be fixed and $\lambda_{\mu_2}^{a+s,e}$ be given by Theorem \ref{Th2-3}-{\rm (i)}. Suppose that $\lambda_{\mu_2}^{a+s,e}>0$.  Let  ${\bf u}^{\mu} $ be a positive steady-state solution of \eqref{model} for small values of $\mu_1$. Recall that 
\begin{equation}\label{KH1}
    \begin{cases}
        0=\mu_1\mathcal{K}u_1^{\mu}+ru_2^{\mu}-(a+s+bu_1^{\mu}+cu_2^{\mu})u_1^{\mu} & x\in\Omega,\cr 
        0=\mu_2\mathcal{K}u_2^{\mu}+su_1^{\mu}-(e+fu_2^{\mu}+gu_1^{\mu})u_2^{\mu} & x\in \Omega.
    \end{cases}
\end{equation}
Solving for $u_1^{\mu}$ in the first equation of \eqref{KH1}, we get 
\begin{equation}\label{KH2}
    u_1^{\mu}=\frac{1}{2b}\left( \sqrt{(a+s+cu_2^{\mu})^2+4b(ru_2^{\mu}+\mu_1\mathcal{K}u_1^{\mu})}-(a+s+cu_2^{\mu})\right).
\end{equation}
Hence, introducing the functions

\begin{equation}
    G(x,\tau):=a(x)+s(x)+c(x)\tau
\end{equation}
and 
\begin{align}\label{KH5}
F(x,\tau,\nu):=&\frac{1}{2b(x)}\left( \sqrt{G^2(x,\tau)+4b(x)(r(x)\tau+\nu)}-G(x,\tau)\right)\cr 
=&\frac{2(\tau r(x)+\nu)}{\sqrt{G^2(x,\tau)+4b(x)(\tau r(x)+\nu)}+G(x,\tau)}
\end{align}
for every $x\in\overline{\Omega}$, $\tau\ge 0$ and $\nu\ge 0$, we have that $u_1^{\mu}=F(\cdot,u_2^{\mu},\mu_1\mathcal{K}u_1^{\mu})$. Observe that $F(x,\tau,0)=H(x,\tau)$ for every $\tau\ge 0$ and $x\in\overline{\Omega}$, where $H$ is defined in \eqref{H-eq}. Hence, it follows from \eqref{KH1} and \eqref{KH2} that 
\begin{equation}\label{KH3}
    0=\mu_2\mathcal{K}u_2^{\mu}+sF(x,u_2^{\mu},\mu_1u_1^{\mu}(x))-(e(x)+f(x)u_2^{\mu}(x)+g(x)F(x,u_2^{\mu},\mu_1\mathcal{K}u_1^{\mu}(x)))u_2^{\mu}(x) \quad \forall\ x\in\overline{\Omega}.
\end{equation}
Note that for every $x\in\overline{\Omega}$ and $\tau\ge 0$, the function $[0,\infty)\ni \nu\mapsto F(x,\tau,\nu)$ is nondecresing. Recall also from Step 1 of the proof of Theorem \ref{TH6} that $$
\|u_i^{\mu}\|_{\infty}\le M^*:=\frac{r_{\max}+s_{\max}}{\min\{b_{\min},f_{\min}\}},
$$
hence, it follows from \eqref{KH3} that 
\begin{equation}\label{KH6}
    0\le \mu_2\mathcal{K}u_2^{\mu}+s(x)F(x,u_2^{\mu},\mu_1M^*)-(e(x)+f(x)u_2^{\mu}(x)+g(x)H(x,u_2^{\mu}))u_2^{\mu}(x) \quad \forall\ x\in\overline{\Omega}
\end{equation}
and 
\begin{equation}\label{KH7}
    0\ge \mu_2\mathcal{K}u_2^{\mu}+s(x)H(x,u_2^{\mu})-(e(x)+f(x)u_2^{\mu}(x)+g(x)F(x,u_2^{\mu},\mu_1M^*))u_2^{\mu}(x) \quad \forall\ x\in\overline{\Omega}.
\end{equation}
Form this point, the proof is divided into three parts.

{\bf Step 1}. In this step, we show that that there is $\mu_1^*>0$ such for every $0<\mu<\mu_1^*$, there is a  unique globally stable solution $\underline{u}_{2}^{\mu}$ of 
\begin{equation}\label{KH8}
    0=\mu_2\mathcal{K}\underline{u}_2^{\mu}+\Big(s(x)\frac{H(x,\underline{u}_2^{\mu})}{\underline{u}_2^{\mu}}-e(x)-f(x)\underline{u}_2^{\mu}(x)-g(x)F(x,\underline{u}_2^{\mu},\mu_1M^*)\Big)\underline{u}_2^{\mu}(x) \quad \forall\ x\in\overline{\Omega}
\end{equation}

Now, observe from \eqref{KH5} that, for every $x\in\overline{\Omega}$, the mapping
$$
(0,\infty)\ni\tau\mapsto \frac{H(x,\tau)}{\tau}=\frac{F(x,\tau,0)}{\tau}=\frac{2 r(x)}{\sqrt{G^2(x,\tau)+4b(x)\tau r(x)}+G(x,\tau)\rd)\bk}
$$
is strictly decreasing and
$$
\lim_{\tau\to 0}\frac{H(x,\tau)}{\tau}=\frac{r(x)}{(a(x)+s(x))}\quad \text{for}\ x \ \text{uniformly on }\ \Omega.
$$
Note that $[0,\infty)\ni\tau\mapsto F(x,\tau,\nu)$ is of class $C^1$ and 
\begin{align*}
    &f(x)+\partial_{\tau}F(x,\tau,\nu)\cr 
    =&f(x)+\frac{\left(\Big(2b(x)r(x)+c(x)G(x,\tau)\Big)-c(x)\sqrt{G^2(x,\tau)+4b(x)(\tau r(x)+\nu)}\right)}{2b(x)\sqrt{G^2(x,\tau)+4b(x)(\tau r(x)+\nu)}} = \mathbb{I}_1(x)+\mathbb{I}_2(x)
\end{align*}
where 

\begin{align}
    \mathbb{I}_1(x,\tau,\nu):=& \frac{\left(\Big(2b(x)r(x)+c(x)G(x,\tau)\Big)-c(x)\sqrt{G^2(x,\tau)+4b(x)\tau r(x)}\right)}{2b(x)\sqrt{G^2(x,\tau)+4b(x)(\tau r(x)+\nu)}}
\end{align} 
and 
\begin{align*}
     & \mathbb{I}_2(x,\tau,\nu)  := f(x)+\frac{c(x)\left(\sqrt{G^2(x,\tau)+4b(x)\tau r(x)}-\sqrt{G^2(x,\tau)+4b(x)(\tau r(x)+\nu)}\right)}{2b(x)\sqrt{G^2(x,\tau)+4b(x)(\tau r(x)+\nu)}}=\cr 
    & f(x)-\frac{4b(x)c(x)\nu}{2b(x)\sqrt{G^2(x,\tau)+4b(x)(\tau r(x)+\nu)}\left(\sqrt{G^2(x,\tau)+4b(x)\tau r(x)}+\sqrt{G^2(x,\tau)+4b(x)(\tau r(x)+\nu)}\right)}
\end{align*}
An easy computation shows that $\mathbb{I}_1(x,\tau,\nu)\ge 0$ for every $x\in\overline{\Omega}$, $\tau\ge 0$, and $\nu\ge 0$, and 
$$ 
\lim_{\nu\to 0}\max_{x\in\overline{\Omega},\tau\in[0,M^*]}\left|\mathbb{I}_2(x,\tau,\nu)-f(x)\right|=0.
$$
Therefore, since $f_{\min}>0$, there is $\tilde{\mu}_1^{*}>0$ such that for every $0<\mu_1\le \tilde{\mu}_1^*$ and $x\in\overline{\Omega}$, the mapping 
$$(0,M^*]\ni\tau \mapsto \tilde{F}(x,\tau,\mu_1):=s(x)\frac{H(x,\tau)}{\tau}-e(x)-f(x)\tau -g(x)F(x,\tau,\mu_1M^*)$$
is strictly decreasing. Observe that $0\le \frac{H(x,\tau)}{\tau}\le \frac{r(x)s(x)}{(a(x)+s(x))}$ for all $x\in\overline{\Omega}$ and $\tau>0$, hence 
$$
\tilde{F}(x,M^*,\mu_1)\le \frac{r(x)s(x)}{a(x)+s(x)}-f(x)M^*\le r(x)-f(x)M^*<0\quad \forall\ x\in\overline{\Omega}, \ 0<\mu_1<\tilde{\mu}^*_1.
$$
Finally, it is easy to see that 
$$ 
\lim_{\mu_1\to0,\tau\to0}\max_{x\in\overline{\Omega}}\left|\tilde{F}(x,\tau,\mu_1)-\Big(\frac{r(x)s(x)}{a(x)+s(x)}-e(x)\Big)\right|=0.
$$
This shows that 
$$
\lim_{\mu_1\to 0}\lambda_p(\mu_2\mathcal{K}+\tilde{F}(x,0,\mu_1))=\lambda_p(\mu_2\mathcal{K}+(\frac{rs}{a+s}-e)\mathcal{I}).
$$
As a result, since $\lambda_p(\mu_2\mathcal{K}+(\frac{rs}{a+s}-e)\mathcal{I}) $ and $\lambda_{\mu_2}^{a+s,e}$ have the same sign by Proposition \ref{Main-prop1}, and $\lambda_{\mu_2}^{a+s,e}>0$, there is $0<\mu_1^*<\tilde{\mu}_1^*$ such that $\lambda_p(\mu_2\mathcal{K}+\tilde{F}(x,0,\mu_1))>0$ for every $0<\mu_1<\mu_1^* $. Hence, it follows from standard arguments on the nonlocal-dispersal equations of Fisher-KPP type that for every $0<\mu_1<\mu_1^*$, there is a unique stable solution  $\underline{u}_2^{\mu}\in X^{++}$ which solves \eqref{KH8}. Moreover, $\|\underline{u}_2^{\mu}\|_{\infty}\le M^*$ for every $0<\mu_1<\mu_1^*$. Furthermore,  since $\mu_1\mapsto F(x,\tau,\mu_1M^*)$ is strictly decreasing on $(0,\mu_1^*)$ for every $x\in\overline{\Omega}$ and $\tau\in[0,M^*]$, then by the comparison principle for cooperative nonlocal equations, we have that $\underline{u}_2^{\mu} $ is decreasing in $\mu_1$. And there is $\underline{u}_2^*\in X^{++}$ such that $\underline{u}_2^{\mu}\to \underline{u}_2^*$. Note that $\underline{u}_2^*$ also solves \eqref{TH8-eq}. Since \eqref{TH8-eq} has a unique solution, then $\underline{u}_2^*=w^*$.

{\bf Step 2.} In this step, we show that that there is $0<\hat{\mu}_1^*<\mu_1^*$ such for every $0<\mu<\hat{\mu}_1^*$, there is  a  unique stable solution $\overline{u}_{2}^{\mu}\ge \underline{u}_2^{\mu}$ of 
\begin{equation}\label{KH9}
0=\mu_2\mathcal{K}\overline{u}_2^{\mu}+s(x)F(x,\overline{u}_2^{\mu},\mu_1M^*)-(e(x)+f(x)\overline{u}_2^{\mu}(x)+g(x)H(x,\overline{u}_2^{\mu}))\overline{u}_2^{\mu}(x) \quad \forall\ x\in\overline{\Omega}.
\end{equation}
The proof follows by a proper modification of the arguments of Step 1. Note also here that $\overline{u}_{2}^{\mu}$ is decreasing in $\mu_1$ and converges uniformly to $w^*$.

{\bf Step 3.} In this step, we complete the proof of the theorem. Indeed, it follows from \eqref{KH6} and \eqref{KH7} that $\underline{u}_{2}^{\mu}<u_2^{\mu}<\overline{u}_2^{\mu}$ for every $0<\mu_1<\hat{\mu}_1^*$, which implies that $\|u_2^{\mu}-w^*\|_{\infty}\to 0$ as $\mu_1\to 0$. This in turn together with \eqref{KH2} implies that $u_1^{\mu}\to H(\cdot,w^*(\cdot)) $ as $\mu_1\to 0$ uniformly on $\Omega$.

\end{proof}

We complete this section with a proof of Theorem \ref{TH9}. 

 \begin{proof}[Proof of Theorem \ref{TH9}] Fix $\mu_2>0$ and suppose that $e\in X^+\setminus\{0\}$. Suppose also that $\lambda_{\mu_2}^{\infty}>0$, where $\lambda_{\mu_2}^{\infty}$ is  given by Theorem \ref{Th2-4}-{\rm(i)}. Let ${\bf u}^{\mu}$, $\mu=(\mu_1,\mu_2)$, be a collection of steady-state solutions of \eqref{model} for sufficiently large values of $\mu_1$.  Hence,  as in Step 1 of proof of Theorem \ref{TH7}, we have 
 \begin{equation*}
     \lim_{\mu_1\to\infty}\left\|u_1^{\mu}-\frac{1}{|\Omega|}\int_{\Omega}u_1^{\mu}\right\|_{\infty}=0.
 \end{equation*}
 Hence, there is a nonnegative number $l^*$ and a sequence $\mu_{1,n}\to \infty$ with $\mu^n=(\mu_{1,n},\mu_2)$ such that 
 \begin{equation}\label{BH7-2}
     \lim_{n\to\infty}\|u_1^{\mu^n}-l^*\|_{\infty}=0.
 \end{equation}
 Now, since $u_2^{\mu^n}$ satisfies 
 \begin{equation}\label{BH7-3}
     0=\mu_2\mathcal{K}(u_2^{\mu^n})+su_1^{\mu^n}-(e+fu_2^{\mu^n}+gu_1^{\mu^n})u_2^{\mu^n} 
 \end{equation}
 and $\lambda_p(\mu_2\mathcal{K}-e\mathcal{I})<0$  since $e\in X^+\setminus\{0\}$, and $\Psi_{\mu_2,e}(\cdot,0,\cdot)$ is strongly positive, then 
 \begin{equation} \label{BH7-4}
0\le  u_2^{\mu^n}=\Psi_{\mu_2,e}\big(\cdot, 0, su_1^{\mu^n}-(fu_2^{\mu^n}+gu_1^{\mu^n})u_2^{\mu^n}\big) \le \Psi_{\mu_2,e}(su_1^{\mu^n})\le \|\Psi_{\mu_2,e}(\cdot,0,\cdot)\| s_{\max} \|u_{1}^{\mu^n}\|_{\infty}\quad \forall\ n\ge 1.
 \end{equation}
 It then follows from \eqref{BH7} and \eqref{BH7-2} that 
 \begin{align*}
0<\lambda_{\mu_2}^{\infty}
=\lim_{n\to\infty}\lambda_p(\mu^n\circ\mathcal{K}+\mathcal{A})\le l^*\max\{1,\|\Psi_{\mu_2,e}(\cdot,0,\cdot)\| s_{\max} \}(\|b\|_{\infty}+\|e\|_{\infty}+\|c\|_{\infty}+\|g\|_{\infty}),
\end{align*}
which implies that $l^*>0$. As a result, defining the function 
$$
\tilde{F}^*(x,\tau)=l^*s(x)-(e(x)+f(x)\tau+l^*g(x))\tau\quad x\in\overline{\Omega},\ \tau\ge 0,
$$
we have that $\tilde{F}^*(\cdot,0)=l^*s\in X^{+}\setminus\{0\}$. Hence, $\lambda_p(\mu_2\mathcal{K}+\tilde{F}^*(\cdot,0))>0$. Moreover, for every $x\in\overline{\Omega}$, the mapping $\tau \mapsto \tilde{F}^*(x,\tau)$ is strictly decreasing and  $ \tilde{F}^*(x,\tau) \bk <0$  for $\tau\ge \sqrt{\frac{2l^*s_{\max}}{f_{\min}}}$. It  follows from standard arguments that  the nonlocal equation 
\begin{equation}
    0=\mu_2\mathcal{K}\tilde{w}+\tilde{F}^*(x,\tilde{w})
\end{equation}
has a unique positive solution $\tilde{w}^*\in X^{++}$. Observing from \eqref{BH7-3} that 
$$
0=\mu_2\mathcal{K}(u_2^{\mu^n})+\tilde{F}^*(x,u_2^{\mu^n})+(s-u_2^{\mu^n}g)(u_1^{\mu^n}-l^*),
$$
and by \eqref{BH7-2} and \eqref{BH7-4}
$$ 
\lim_{n\to\infty}\|(s-u_2^{\mu^n}g)(u_1^{\mu^n}-l^*)\|_{\infty}=0,
$$
we can employ a perturbation argument to conclude that $u_2^{\mu^n}\to \tilde{w}^*$ as $n\to\infty$, uniformly in $\Omega$. Finally, integrating over $\Omega$ the  equation
$$ 
0=\mu_{1,n}\mathcal{K}(u_1^{\mu^n})+ru_2^{\mu^n}-(a+s+bu_1^{\mu^n}+cu_2^{\mu^n})u_1^{\mu^n}
$$
 and letting $n\to\infty$ in the resulting equation yields
$$
0=\int_{\Omega}r\tilde{w}^*-l^*\int_{\Omega}((a+s)+bl^*+c\tilde{w}^*).
$$
Therefore, $(l^*,\tilde{w}^*)$ satisfies \eqref{TH9-eq} and ${\bf u}^{\mu^n}\to(l^*,\tilde{w}^*) $ as $n\to\infty$, uniformly in $\Omega$.
   
 \end{proof}
\subsection*{Acknowledgement.} We sincerely thank the anonymous reviewer for the helpful comments and suggestions.
 

\end{document}